\documentclass{amsart}

\usepackage{amssymb}
\usepackage{amsthm}
\usepackage{amsmath}
\usepackage{graphicx}
\usepackage{subcaption}
\usepackage[color=white]{todonotes}
\usepackage{float}
\usepackage{enumitem}
\usepackage{bm}
\usepackage{xcolor}

\title[A generalized FEM for the strongly damped wave equation]{A generalized finite element method for the strongly damped wave equation with rapidly varying data}
\author[Per Ljung]{Per Ljung\textsuperscript{1}}
\author[Axel M{\aa}lqvist]{Axel M{\aa}lqvist\textsuperscript{1}} 
\author[Anna Persson]{Anna Persson\textsuperscript{2}}

\newtheorem{definition}{Definition}[section]
\newtheorem{thm}[definition]{Theorem}
\newtheorem{lemma}[definition]{Lemma}
\newtheorem{corollary}[definition]{Corollary}
\newtheorem*{assump}{\sc{Assumptions}}
\theoremstyle{remark}
\newtheorem{remark}[definition]{Remark}
\numberwithin{equation}{section}

\DeclareMathOperator*{\esssup}{ess\,sup}
\DeclareMathOperator*{\essinf}{ess\,inf}

\DeclareMathOperator{\supp}{supp}
\DeclareMathOperator{\diam}{diam}

\newcommand{\ddt}{\bar\partial_t}

\newcommand{\ms}{\mathrm{ms}}
\newcommand{\f}{\mathrm{f}}
\newcommand{\ulod}[1]{u_{\mathrm{lod}}^{#1}}
\newcommand{\ulodk}[1]{u_{\mathrm{lod}, k}^{#1}}

\newcommand{\Z}[1]{Z^{#1}_h}
\newcommand{\Zlod}[1]{Z_\mathrm{lod}^{#1}}

\newcommand{\vertiii}[1]{{\left\vert\kern-0.25ex\left\vert\kern-0.25ex\left\vert #1 
		\right\vert\kern-0.25ex\right\vert\kern-0.25ex\right\vert}}

\usepackage{hyperref}

\begin{document}
	
	\begin{abstract}
		We propose a generalized finite element method for the strongly damped wave equation with highly varying coefficients. The proposed method is based on the localized orthogonal decomposition introduced in \cite{MalqvistPeterseim}, and is designed to handle independent variations in both the damping and the wave propagation speed respectively. The method does so by automatically correcting for the damping in the transient phase and for the propagation speed in the steady state phase. Convergence of optimal order is proven in $L_2(H^1)$-norm, independent of the derivatives of the coefficients. We present numerical examples that confirm the theoretical findings.
	\end{abstract}
	
	\maketitle
	\footnotetext[1]{Department of Mathematical Sciences, Chalmers University of Technology and University of Gothenburg, SE-412 96 Gothenburg, Sweden.}
	\footnotetext[2]{Department of Mathematics, KTH Royal Institute  of  Technology,, SE-100 44 Stockholm, Sweden.}
\section{Introduction}

This paper is devoted to the study of numerical solutions to the strongly damped wave equation with highly varying coefficients. The equation takes the general form 
\begin{equation}
    \ddot{u} - \nabla \cdot (A\nabla \dot{u} + B\nabla u) = f,
\end{equation}
on a bounded domain $\Omega$. Here, $A$ and $B$ represent the system's damping and wave propagation respectively, $f$ denotes the source term, and the solution $u$ is a displacement function. This equation commonly appears in the modelling of viscoelastic materials, where the strong damping $-\nabla \cdot A\nabla \dot{u}$ arises due to the stress being represented as the sum of an elastic part and a viscous part \cite{ViscoelasticExample1, Attractor1}. Viscoelastic materials have several applications in engineering, including noise dampening, vibration isolation, and shock absorption (see \cite{ViscoBook} for more applications). In particular, in multiscale applications, such as modelling of porous medium or composite materials, $A$ and $B$ are both rapidly varying.

There have been much recent work regarding strongly damped wave equations. For instance, well-posedness of the problem is discussed in \cite{Wp1, WP2, WP3}, asymptotic behavior in \cite{Asym1, Asym2, Asym3, Asym4} solution blowup in \cite{blowup1, blowup2}, and decay estimates in \cite{SDWEref18}. In particular, FEM for the strongly damped wave equation has been analyzed in \cite{RitzVolterra} using the Ritz--Volterra projection, and \cite{FEMdampedwavehom} uses the classical Ritz-projection in the homogeneous case with Rayleigh damping. In these papers, convergence of optimal order is shown. However, in the case of piecewise linear polynomials, the convergence relies on at least $H^2$-regularity in space. Consequently, since the $H^2$-norm depends on the derivatives of the coefficients, the error is bounded by  $\|u\|_{H^2} \sim \max(\varepsilon^{-1}_A, \varepsilon^{-1}_B)$ where $\varepsilon_A$ and $\varepsilon_B$ denote the scales at which $A$ and $B$ vary respectively. The convergence order is thus only valid when the mesh width $h$ fulfills $h < \min(\varepsilon_A, \varepsilon_B)$. In other words, we require a mesh that is fine enough to resolve the variations of $A$ and $B$, which becomes computationally challenging. This type of difficulty is common for equations with rapidly varying data, an issue for which several numerical methods have been developed (see e.g.\ \cite{MS1, MS2, MS3, MS4, MS5}). None of these methods are however applicable to the strongly damped wave equation, where two different multiscale coefficients have to be dealt with. In this paper, we propose a novel multiscale method based on the localized orthogonal decomposition (LOD) method. 

The LOD method is based on the variational multiscale method presented in \cite{VariationalMultiscaleMethod}. It was first introduced in \cite{MalqvistPeterseim}, and has since then been further developed and analyzed for several types of problems (see e.g.\ \cite{lodbook, Parabolic, Wave, SemiLinear, boundaryproblems}). In particular, \cite{quadeigenvalue} studies the LOD method for quadratic eigenvalue problems, which correspond to time-periodic wave equations with weak damping. The main idea of the method is based on a decomposition of the solution space into a coarse and a fine part. The decomposition is done by defining an interpolant that maps functions from an infinite dimensional space into a finite dimensional FE-space. In this way, the kernel of the interpolant captures the finescale features that the coarse FE-space misses, and hence defines the finescale space. Subsequently, one may use the orthogonal complement to this finescale space with respect to a problem-dependent Ritz-projection as a modified FE-space. In the case of time-dependent problems, the LOD method performs particularly well in the sense that the modified FE-space only needs to be computed once, and can then be re-used in each time step.


Multiscale methods, as the localized orthogonal decomposition, are usually designed to handle problems with a single multiscale coefficient. In this sense, the strongly damped wave equation is different, as an extra coefficient appears due to the strong damping. Hence, one of the main challenges for the novel method is how to incorporate the finescale behavior of both coefficients in the computation. Nevertheless, it should be noted that existing multiscale methods are applicable for some special cases of this equation. An example is the case of Rayleigh damping where the coefficients are proportional to each other. Other examples are the steady state case, the transient phase in which the solution evolves rapidly in time, as well as the case of weak damping where no spatial derivatives are present on the damping term.

In this paper we present a generalized finite element method (GFEM) for solving the strongly damped wave equation. The method uses both the damping and diffusion coefficients to construct a generalized finite element space, similar to those in e.g.\ \cite{MalqvistPeterseim, Parabolic}. The solution is then evaluated in this space, but to account for the time dependence, an additional correction is added to it. However, this correction is evaluated on the fine scale, and thus expensive to compute. To overcome this issue, we prove spatial exponential decay for the corrections so that we can restrict the problems to patches in a similar manner as for the modified basis functions in \cite{MalqvistPeterseim}. The effect of the proposed method is that the multiscale basis compensates for the damping early on in the simulation when it is dominant and then gradually starts to compensate for the wave propagation which is dominant at steady state. This is done seamlessly and automatically by the method. Furthermore, we prove optimal order convergence in $L_2(H^1)$-norm for this method. Following this, we show that it is sufficient to compute the finescale corrections for only a few time steps by applying reduced basis (RB) techniques. For related work on RB methods, see e.g.\ \cite{rbref1, rbref2, rbref3}, and for an introduction to the topic we refer to \cite{rbbook}. 

The outline of the paper is as follows: In Section \ref{sec:weakform} we present the weak formulation and classical FEM for the strongly damped wave equation, along with necessary assumptions. Section \ref{sec:gfem} is devoted to the generalized finite element method and its localization procedure. In Section \ref{sec:errorestimates} error estimates for the method are proven. Section \ref{sec:rb} covers the details of the RB approach, and finally in Section \ref{sec:numexamples} we illustrate numerical examples that confirm the theory derived in this paper.

\section{Weak formulation and classical FEM}
\label{sec:weakform}

We consider the wave equation with strong damping of the following form
\begin{alignat}{2}
    \ddot{u} - \nabla \cdot (A\nabla \dot{u} + B\nabla u) &= f, \quad &&\text{in $\Omega \times (0,T]$}, \label{eq:sdwe}\\
    u&= 0, \quad &&\text{on $\Gamma \times (0,T]$}, \label{eq:dbc}\\
    u(0) &= u_0, \quad &&\text{in $\Omega$}, \label{eq:init1} \\
    \dot{u}(0) &= v_0 \quad &&\text{in $\Omega$}, \label{eq:init2}
\end{alignat}
where $T > 0$ and $\Omega$ is a polygonal (or polyhedral) domain in $\mathbb{R}^d, \ d=2,3,$ and $\Gamma := \partial \Omega$. The coefficients $A$ and $B$ describe the damping and propagation speed respectively, and $f$ denotes the source function of the system. We assume $A=A(x)$,  $B=B(x)$ and $f=f(x,t)$, i.e.\ the multiscale coefficients are independent of time. 

Denote by $H^1_0(\Omega)$ the classical Sobolev space with norm 
\begin{equation*}
    \|v\|^2_{H^1(\Omega)} = \|v\|^2_{L_2(\Omega)} + \|\nabla v\|^2_{L_2(\Omega)}
\end{equation*}
whose functions vanish on $\Gamma$. Moreover, let $L_p(0,T; \mathcal{B})$ be the Bochner space with norm 
\begin{align*}
    \|v\|_{L_p(0,T; \mathcal{B})} &= \bigg( \int_0^T \|v\|_{\mathcal{B}}^p\, \mathrm{d}t \bigg)^{1/p}, \ \ p\in [1,\infty), \\
    \|v\|_{L_\infty(0,T; \mathcal{B})} &= \esssup_{t\in [0,T]} \|v\|_{\mathcal{B}},
\end{align*}
where $\mathcal{B}$ is a Banach space with norm $\|\cdot\|_\mathcal{B}$. In this paper, following assumptions are made on the data.

\begin{assump}
	The damping and propagation coefficients $A,B\in L_\infty(\Omega, \mathbb{R}^{d\times d})$ are symmetric and satisfy	
	\begin{align*}
		0 < \alpha_- := \essinf_{x\in \Omega} \inf_{v\in \mathbb{R}^d\backslash \{0\}} \frac{A(x)v\cdot v}{v\cdot v} < \esssup_{x\in \Omega} \sup_{v\in \mathbb{R}^d\backslash \{0\}} \frac{A(x)v\cdot v}{v\cdot v} =: \alpha_+ < \infty, \\
		0 < \beta_- := \essinf_{x\in \Omega} \inf_{v\in \mathbb{R}^d\backslash \{0\}} \frac{B(x)v\cdot v}{v\cdot v} < \esssup_{x\in \Omega} \sup_{v\in \mathbb{R}^d\backslash \{0\}} \frac{B(x)v\cdot v}{v\cdot v} =: \beta_+ < \infty.
	\end{align*}
	In addition, we assume that $f\in L_\infty([0,T];L_2(\Omega))$ and $\dot{f}\in L_2([0,T];L_2(\Omega))$.
\end{assump}

For the spatial discretization, let $\{\mathcal{T}_h\}_{h>0}$ denote a family of shape regular elements that form a partition of the domain $\Omega$. For an element $K \in \mathcal{T}_h$, let the corresponding mesh size be defined as $h_K := \diam(K)$, and denote the largest diameter of the partition by $h := \max_{K\in \mathcal{T}_h} h_K$. We now define the classical FE-space using continuous piecewise linear polynomials as
\begin{equation*}
    S_h := \{v\in \mathcal{C}(\bar{\Omega}) :\text{$v\big|_{\Gamma} = 0$, $v\big|_K$ is a polynomial of partial degree $\leq 1$, $\forall K\in \mathcal{T}_h$}\},
\end{equation*}
and let $V_h = S_h\cap H^1_0(\Omega)$. The semi-discrete FEM becomes: find $u_h(t) \in V_h$ such that
\begin{equation}
    (\ddot{u}_h,v) + a(\dot{u}_h,v) + b(u_h,v) = (f,v), \quad \forall v\in V_h, \ t > 0,
    \label{eq:semidiscreteweakform}
\end{equation}
with initial values $u_h(0) = u_{h,0}$ and $\dot{u}_h(0) = v_{h,0}$ where $u_{h,0}, v_{h,0} \in V_h$ are appropriate approximations of $u_0$ and $v_0$ respectively. Here $(\cdot,\cdot)$ denotes the usual $L_2$-inner product, $a(\cdot,\cdot) = (A\nabla\cdot,\nabla\cdot)$, and $b(\cdot,\cdot) = (B\nabla\cdot,\nabla\cdot)$.

For the temporal discretization, let $0 = t_0 < t_1 < ... < t_N = T$ be a uniform partition with time step $t_n - t_{n-1} = \tau$. The time step here is chosen uniformly for simplicity, but the choice of varying time step is still viable. We apply a backward Euler scheme to get the fully discrete system: find $u^n_h \in V_h$ such that
\begin{equation}
    (\ddt^2 u^n_h, v) + a(\ddt u^n_h, v) + b(u^n_h, v) = (f^n,v), \quad \forall v\in V_h
    \label{eq:fullydiscrete}
\end{equation}
for $n\geq 2$. Here, the discrete derivative is defined as $\ddt u^n_h = (u^n_h - u^{n-1}_h)/\tau$.

For results on regularity and error estimates for the FEM solution of the strongly damped wave equation, we refer to \cite{FEMdampedwavehom}. Moreover, existence and uniqueness of a solution to \eqref{eq:fullydiscrete} is guaranteed by Lax--Milgram.

In the analysis, we use the notations $\|\cdot\|^2_a := a(\cdot,\cdot)$,  $\|\cdot\|^2_b := b(\cdot,\cdot)$, as well as $\vertiii{\cdot}^2 = \tilde{a}(\cdot, \cdot) := a(\cdot, \cdot) + \tau b(\cdot, \cdot)$, and the fact that these are equivalent with the $H^1$-norm. That is, there exist positive constants $C_a, C_b, C_{\tilde{a}}, c_a, c_b, c_{\tilde{a}} \in \mathbb{R}$, such that
\begin{alignat}{2}
c_a\|v\|^2_{H^1} &\leq \|v\|^2_a \leq C_a\|v\|^2_{H^1}, \quad &&\forall v\in H^1(\Omega), \notag \\
c_b\|v\|^2_{H^1} &\leq \|v\|^2_b \leq C_b\|v\|^2_{H^1}, \quad  &&\forall v\in H^1(\Omega), \label{norm_equivalence}\\
c_{\tilde{a}}\|v\|^2_{H^1} &\leq \vertiii{v}^2 \leq C_{\tilde{a}}\|v\|^2_{H^1}, \quad &&\forall v\in H^1(\Omega). \notag
\end{alignat}

\begin{thm}\label{thm:energy_bounds}
	The solution $u^n_h$ to \eqref{eq:fullydiscrete} satisfies the following bounds
	\begin{alignat}{2}
	\|\ddt u^n_h\|^2_{L_2} + \sum_{j=2}^n \tau\|\ddt u^j_h\|^2_{H^1} + \|u^n_h\|^2_{H^1} &\leq C\sum_{j=2}^n \tau\|f^j\|^2_{H^{-1}} \label{energy_bound1} + C(\|\ddt u^1_h\|^2_{L_2} + \|u^1_h\|^2_{H^1}), \\
	\sum_{j=2}^n\tau\|\ddt^2 u^j_h\|^2_{L_2} + \|\ddt u^n_h\|^2_{H^1} &\leq C\sum_{j=2}^n \tau\|f^j\|^2_{L_2} \label{energy_bound2} + C(\|\ddt u^1_h\|^2_{H^1} + \|u^1_h\|^2_{H^1}),.
	\end{alignat}
	for $n\geq 2$.
	
%
\end{thm}

\begin{proof}
	To prove \eqref{energy_bound1}, choose $v=\tau \ddt u^n_h$ in \eqref{eq:fullydiscrete} to get
	\begin{align}\label{energy_bounds_eq1}
	\tau(\ddt^2 u^n_h, \ddt u^n_h) + \tau\|\ddt u^n_h\|^2_a + \tau b(u^n_h, \ddt u^n_h) = \tau (f^n,\ddt u^n_h).
	\end{align}
	Due to Cauchy--Schwarz and Young's inequality we have the following lower bound
	\begin{align*}
	\tau(\ddt^2 u^n_h, \ddt u^n_h) = \|\ddt u^n_h\|^2_{L_2} - (\ddt u^{n-1}_h, \ddt u^n_h) \geq \frac{1}{2}\|\ddt u^n_h\|^2_{L_2} - \frac{1}{2}\|\ddt u^{n-1}_h\|^2_{L_2},
	\end{align*}
	and similarly
	\begin{align*}
	\tau b(u^n_h, \ddt u^n_h) \geq \frac{1}{2}\|u^n_h\|^2_{b} - \frac{1}{2}\|u^{n-1}_h\|^2_{b}.
	\end{align*}
	Similar bounds will be used repeatedly throughout the paper. Summing \eqref{energy_bounds_eq1} over $n$ gives
	\begin{align*}
	\frac{1}{2}\|\ddt u^n_h\|^2_{L_2} - \frac{1}{2}\|\ddt u^{1}_h\|^2_{L_2} + \sum_{j=2}^n \tau \|\ddt u^j_h\|^2_a &+ \frac{1}{2}\|u^n_h\|^2_{b} - \frac{1}{2}\|u^{1}_h\|^2_{b} \leq \sum_{j=2}^n \tau\|f^j\|_{H^{-1}}\|\ddt u^j_h\|_{H^1}.
	\end{align*}
	Using the equivalence of the norms \eqref{norm_equivalence}, Cauchy--Schwarz and Young's (weighted) inequality to subtract $\sum_{j=2}^n \tau\|\ddt u^j_h\|^2_{H^1}$ from both sides, we get exactly \eqref{energy_bound1}.
	
	The proof of \eqref{energy_bound2} is similar. We choose $v=\tau\ddt^2u^n_h$ in \eqref{eq:fullydiscrete} and sum over $n$ to get
	\begin{align*}
	\sum_{j=2}^n\tau\|\ddt^2 u^j_h\|^2_{L_2} + \frac{1}{2}\|\ddt u^n_h\|^2_a - \frac{1}{2}\|\ddt u^1_h\|^2_a &+ \sum_{j=2}^n\tau b(u^j_h, \ddt^2 u^j_h) \leq \sum_{j=2}^n \tau \|f^j\|_{L_2}\|\ddt^2 u^j_h\|_{L_2}.
	\end{align*}
	For the sum involving the bilinear form $b(\cdot,\cdot)$ we use summation by parts to get
	\begin{align*}
	\sum_{j=2}^n\tau b(u^j_h, \ddt^2 u^j_h) = \sum_{j=3}^n -\tau b(\ddt u^j_h, \ddt u^{j-1}_h) - b(u^2_h, \ddt u^1_h) + b(u^n_h, \ddt u^n_h). 
	\end{align*}
	Using \eqref{energy_bound1}, the equivalence of the norms \eqref{norm_equivalence}, and Young's weighted inequality we have
	\begin{align*}
	&\left|\sum_{j=3}^n  \tau b(\ddt u^j_h, \ddt u^{j-1}_h) + b(u^2_h, \ddt u^1_h) - b(u^n_h, \ddt u^n_h)\right| \\&\qquad \quad\leq C \sum_{j=3}^n \tau \|\ddt u^j_h\|^2_{H^1} + C(\|u^2_h\|^2_{H^1} + \|\ddt u^1_h\|^2_{H^1}) + C\|u^n_h\|^2_{H^1} + C_\epsilon\|\ddt u^n_h\|^2_{a} \\
	&\qquad \quad \leq C\sum_{j=2}^n \tau\|f^j\|^2_{H^{-1}} + C(\|\ddt u^1_h\|^2_{H^1} + \|u^1_h\|^2_{H^1}) + C_\epsilon\|\ddt u^n_h\|^2_{a}.
	\end{align*}
	Since $C_\epsilon$ can be made arbitrarily small, it can be kicked to the left hand side. Using that $\|f^j\|^2_{H^{-1}} \leq C\|f^j\|^2_{L_2}$ we deduce \eqref{energy_bound2}.

\end{proof}

\section{Generalized finite element method}
\label{sec:gfem}

This section is dedicated to the development of a multiscale method based on the framework of the standard LOD. First of all, we introduce some notation for the discretization. Let $V_H$ be a FE-space defined analogously to $V_h$ in previous section, but with larger mesh size $H > h$. Moreover, we assume that corresponding family of partitions $\{\mathcal{T}_H\}_{H>h}$ is, in addition to shape-regular, also quasi-uniform. Denote by $\mathcal{N}$ the set of interior nodes of $V_H$ and by $\lambda_x$ the standard hat function for $x\in \mathcal{N}$, such that $V_H = \text{span}(\{\lambda_x\}_{x\in\mathcal{N}})$. Finally, we make the assumption that $\mathcal{T}_h$ is a refinement of $\mathcal{T}_H$, such that $V_H \subseteq V_h$.

\subsection{Ideal method}

To define a generalized finite element method for our problem, we aim to construct a multiscale space $V_\ms$ of the same dimension as $V_H$, but with better approximation properties. For the construction of such a multiscale space, let $I_H:V_h \rightarrow V_H$ be an interpolation operator that has the projection property $I_H = I_H \circ I_H$ and satisfies 
\begin{equation}
    H^{-1}_K\|v-I_Hv\|_{L_2(K)} + \|\nabla I_Hv\|_{L_2(K)} \leq C_I \|\nabla v\|_{L_2(N(K))}, \quad \forall K\in \mathcal{T}_H, \ v\in V_h,
    \label{eq:intestimate}
\end{equation}
where $N(K) := \{K' \in \mathcal{T}_H : \overline{K'}\cap \overline{K} \neq \emptyset\}.$ Furthermore, for a shape-regular and quasi-uniform partition, the estimate \eqref{eq:intestimate} can be summed into the global estimate 
\begin{equation*}
    H^{-1}\|v-I_H\|_{L_2(\Omega)} + \|\nabla I_H v\|_{L_2(\Omega)} \leq C_\gamma \| \nabla v\|_{L_2(\Omega)},
\end{equation*}
where $C_\gamma$ depends on the interpolation constant $C_I$ and the shape regularity parameter defined as
\begin{equation*}
    \gamma := \max_{K\in \mathcal{T}_H} \gamma_K, \ \text{ where } \ \gamma_K = \frac{\text{diam}(B_K)}{\text{diam}(K)}.
\end{equation*}
Here $B_K$ denotes the largest ball inside $K$. A commonly used example of such an interpolant is $I_H = E_H \circ \Pi_H$, where $\Pi_H$ is the piecewise $L_2$-projection onto $P_1(\mathcal{T}_H)$, the space of functions that are affine on each triangle $K \in \mathcal{T}_H$, and $E_H : P_1(\mathcal{T}_H) \rightarrow V_H$ is an averaging operator that, to each free node $x\in \mathcal{N}$, assigns the arithmetic mean of corresponding function values on intersecting elements, i.e.
\begin{equation*}
    (E_H(v))(x) = \frac{1}{\text{card}\{K \in \mathcal{T}_H : x \in \overline{K}\}} \sum_{K \in \mathcal{T}_H : x \in \overline{K}} v\big|_K(x).
\end{equation*}
For more discussion regarding possible choices of interpolants, see e.g.\ \cite{EfficientImplementation} or \cite{expdecayoffs}.

Let the space $V_\f$ be defined by the kernel of the interpolant, i.e.
\begin{equation*}
    V_\f = \ker(I_H) = \{v \in V_h : I_Hv = 0\}.
\end{equation*}
That is, $V_\f$ is a finescale space in the sense that it captures the features that are excluded from the coarse FE-space. This consequently leads to the decomposition
\begin{equation*}
    V_h = V_H \oplus V_\f,
\end{equation*}
such that every function $v\in V_h$ has a unique decomposition $v = v_H + v_\f$, where $v_H \in V_H$ and $v_\f \in V_\f$.

In the case of the LOD method for the standard wave equation (see \cite{Wave}), one considers a Ritz-projection based solely on the $B$-coefficient to construct a multiscale space. Instead, the goal is to define a multiscale space based on the inner product $a(\cdot,\cdot) + \tau b(\cdot, \cdot)$ (for a fixed $\tau$) and add additional correction to account for the time-dependency. This particular choice of scalar product comes from the backward Euler time-stepping formulation and both simplifies the analysis and is more natural in the implementation. Another possibility is to choose $a(\cdot, \cdot)$ as scalar product. For $v\in V_H$, we consider the Ritz-projection $R_\f:V_H\rightarrow V_\f$ defined by
\begin{equation*}
    a(R_\f v, w) + \tau b(R_\f v, w) = a(v,w) + \tau b(v,w), \quad \forall w\in V_\f.
\end{equation*}
Using this projection, we may define the multiscale space $V_\ms := V_H - R_\f V_H$ such that
\begin{equation}\label{MSspace}
    V_h = V_\ms \oplus V_\f, \quad \text{ and } \quad a(v_\ms, v_\f) + \tau b(v_\ms, v_\f) = 0.
\end{equation}
Note that $\dim(V_\ms) = \dim(V_H)$, and hence we can view $V_\ms$ as a modified coarse space that contains finescale information of $A$ and $B$. Next, we may use the Ritz-projection to define the basis functions for the space $V_\ms$. For $x\in \mathcal{N}$, denote by $\phi_x := R_\f\lambda_x \in V_\f$ the solution to the (global) corrector problem
\begin{equation}
    a(\phi_x, w) + \tau b(\phi_x, w) = a(\lambda_x, w) + \tau b(\lambda_x,w), \quad \forall w\in V_\f.\label{eq:basisproj}
\end{equation}
We can now construct our basis for $V_\ms$ as $\{\lambda_x - \phi_x\}_{x\in \mathcal{N}}$ which includes the behavior of the coefficients. For an illustration of the Ritz-projected hat function, as well as the modified basis function for $V_\ms$, see Figure \ref{fig:newbasis}. 

\begin{figure}
	\centering
	\begin{subfigure}[b]{0.4\textwidth}
		\includegraphics[width=\textwidth]{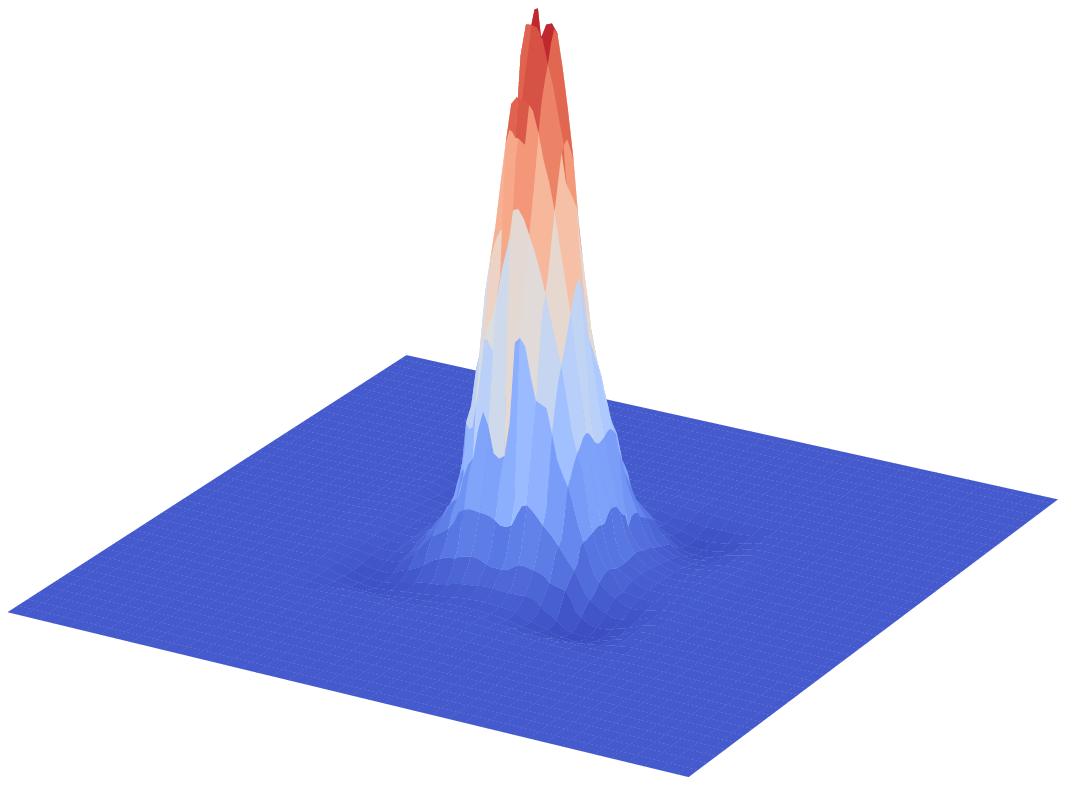}
		\caption{$\lambda_x - \phi_x$.}
		\label{fig:msbasis}
	\end{subfigure}
	~ 
	\begin{subfigure}[b]{0.4\textwidth}
		\includegraphics[width=\textwidth]{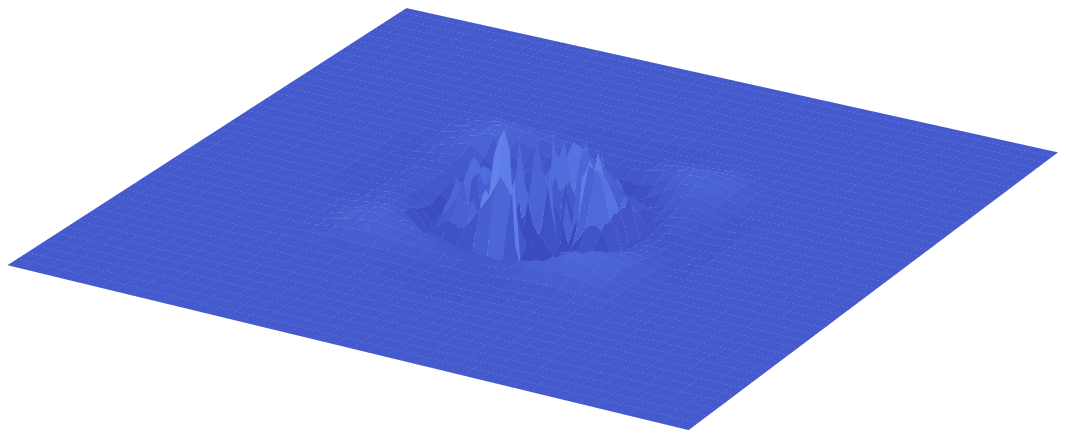}
		\caption{$\phi_x$.}
		\label{fig:corrbasis}
	\end{subfigure}
	\caption{The modified basis function $\lambda_x - \phi_x$ and the Ritz-projected hat function $\phi_x$.}\label{fig:newbasis}
\end{figure}

We may now formulate our ideal (but impractical) method. Since the solution space can be decomposed as $V_h = V_\ms \oplus V_\f$, the idea is to solve a coarse scale problem in $V_\ms$, and then add additional correction from a problem on the fine scale. The method reads: find $\ulod{n} = v^n + w^n$, where $v^n \in V_\ms$ and $w^n\in V_\f$ such that
\begin{alignat}{2}
\tau (\ddt^2v^n, z) + a(v^n, z) + \tau b(v^n, z) &= \tau(f^n,z) + a(\ulod{n-1}, z),& \quad &\forall z \in V_\ms, \label{eq_GFEM_full_1} \\
a(w^n, z) + \tau b(w^n, z) &= a(\ulod{n-1}, z), & & \forall z \in V_\f,\label{eq:GFEM_full_2}
\end{alignat}
for $n\geq 2$ with initial data $\ulod{0}= u^0_h \in V_\ms$ and $\ulod{1} = u^1_h \in V_\ms$. The initial data is chosen in $V_\ms$ to simplify the implementation of the finescale correctors. We further discuss this choice in Section \ref{sec:initialdata}.

\begin{remark}
	Note that in \eqref{eq:GFEM_full_2}, we do not take neither the source function nor the second derivative into account. This is because we can subtract an interpolant within the $L_2$-product, so that corresponding error converges at the same order as the method itself. Moreover, the $v^n$-part and $w^n$-part have been excluded from the bilinear form $a(\cdot,\cdot) + \tau b(\cdot, \cdot)$ in \eqref{eq_GFEM_full_1} and \eqref{eq:GFEM_full_2} respectively, due to the orthogonality between $V_\ms$ and $V_\f$.
\end{remark}

Note that the multiscale space $V_\ms$ is created using \eqref{eq:basisproj} with small $\tau$. Thus, the $A$-coefficient dominates the system for short times. Moreover, we note from \eqref{eq:GFEM_full_2} that for $N$ large enough, we reach a steady state so that $w^N \approx w^{N-1}$ and $v^N\approx v^{N-1}$. We get for $z\in V_\f$
\begin{align*}
a(w^N,z) + \tau b(w^N,z) &\approx a(\ulod{N},z) = a(v^N, z) + a(w^N, z) = -\tau b(v^N, z) + a(w^N, z),
\end{align*}
due to the orthogonality. Hence, by rearranging terms we have that
\begin{align*}
	b(v^N, z) + b(w^N, z) = b(\ulod{N}, z) \approx 0,
\end{align*}
which shows that the solution converges to a state where it is orthogonal with respect to $B$.

\subsection{Localized method}
\label{sec:loc}

The method we have considered so far is based on the global projection \eqref{eq:basisproj} onto the finescale space $V_\f$, which results in a large linear system that is expensive to solve. Moreover, the basis correctors yield a global support that makes the linear system \eqref{eq_GFEM_full_1} not sparse, but dense. Hence, we wish to localize the computations onto coarse grid patches in order to yield a sparse matrix system. 

To localize the corrector problem, we first introduce the patches to which the support of each basis function is to be restricted. For $\omega \subset \Omega$, let $N(\omega) := \{K\in \mathcal{T}_H : \overline{K}\cap \overline{\omega} \neq \emptyset\}$, and define a patch $N^k(\omega)$ of size $k$ as
\begin{align*}
N^1(\omega) &:= N(\omega), \\
N^k(\omega) &:= N(N^{k-1}(\omega)), \ \text{ for $k \geq 2$}.
\end{align*}
Given these coarse grid patches, we may restrict the finescale space $V_\f$ to them by defining 
\begin{equation*}
V_{\f, k}^\omega := \{v\in V_\f : \supp(v) \subseteq N^k(\omega)\},
\end{equation*}
for a subdomain $\omega \subset \Omega$. In particular, we will commonly use $\omega= T\in \mathcal{T}_H$ and $\omega= x\in \mathcal{N}$.

Next, define the element restricted Ritz-projection $R^T_\f$ such that $R^T_\f v\in V_\f$ is the solution to the system 
\begin{equation*}
a(R^T_\f v, z) + \tau b(R^T_\f v, z) = \int_T (A+\tau B)\nabla v\cdot \nabla z\, \mathrm{d}x, \quad \forall z\in V_\f.
\end{equation*}
Note that we may construct the global Ritz-projection as the sum
\begin{equation*}
R_\f v = \sum_{T\in \mathcal{T}_H} R^T_\f v.
\end{equation*}
For $k\in \mathbb{N}$, we may restrict the projection to a patch by letting $R^T_{\f,k}:V_H\rightarrow V_{\f, k}^T$ be such that $R^T_{\f,k}v \in V_{\f, k}^T$ solves
\begin{equation*}
a(R^T_{\f,k}v, z) + \tau b(R^T_{\f,k}v, z) = \int_T (A+\tau B)\nabla v\cdot \nabla z\, \mathrm{d}x, \quad \forall z\in V_{\f, k}^T.
\end{equation*}
By summation we yield the corresponding global version as
\begin{equation*}
R_{\f,k}v = \sum_{T\in \mathcal{T}_H} R^T_{\f,k}v. 
\end{equation*}
Finally, we may construct a localized multiscale space as $V_{\ms,k} := V_H - R_{\f,k}V_H$, spanned by $\{\lambda_x - R_{\f,k}\lambda_x\}_{x\in \mathcal{N}}$. 

In order to justify the act of localization, it is required that a corrector $\phi_x$ vanishes rapidly outside an area of its corresponding node $x$. Indeed, the following theorem from \cite{lodbook} shows that the corrector $\phi_x$ satisfy an exponential decay away from its node, making the localization procedure viable.
\begin{thm}
	There exists a constant $c\geq (8C_I\gamma (2+C_I))^{-1}$, that only depends on the mesh constant $\gamma$, such that for any $T\in \mathcal{T}_H$ and any $v\in H^1_0(\Omega)$ the solution $\phi \in V_\f$ of the variational problem
	\begin{equation*}
		\tilde{a}(\phi, w) = \int_T (\tilde{A}\nabla v) \cdot \nabla w\, \mathrm{d}x, \quad \forall w\in V_\f
	\end{equation*}
	satisfies
	\begin{equation*}
		\|\tilde{A}^{1/2}\nabla \phi\|_{L_2(\Omega\backslash N^k(T))} \leq \sqrt{2}\exp\big(-c\tfrac{\alpha_- + \tau \beta_-}{\alpha_+ + \tau\beta_+}k\big)\|\tilde{A}^{1/2}\nabla v\|_{L_2(T)}, \quad \forall k\in \mathbb{N},
	\end{equation*}
	where $\tilde{A} = A+\tau B$.
	\label{thm:expdecay}
\end{thm}

With the space $V_{\ms,k}$ defined, we are able to localize the computations on the coarse scale system in \eqref{eq_GFEM_full_1} by replacing the multiscale space by its localized counterpart. It remains to localize the computations of the finescale system in \eqref{eq:GFEM_full_2}, which equivalently can be written as
\begin{equation*}
		a(\ddt w^n, z) + b(w^n, z) = \frac{1}{\tau} a(v^{n-1}, z).
\end{equation*}
We replace the right hand side by its localized version $v^{n-1}_k \in V_{\ms, k}$ and note that $v^{n-1}_k = \sum_{x\in \mathcal{N}} \alpha^{n-1}_x (\lambda_x - R_{\f,k}\lambda_x)$. Thus, we seek our localized finescale solution as $w^n_k = \sum_{x\in \mathcal{N}}w^n_{k,x}$, where $w^n_{k,x} \in V_{\f,k}^x$ solves
\begin{equation}
	a(\ddt w^n_{k,x}, z) + b(w^n_{k,x}, z) = \frac{1}{\tau} a(\alpha^{n-1}_x(\lambda_x - R_{\f,k} \lambda_x), z), \quad \forall z\in V_{\f,k}^x,
	\label{eq:wnkx}
\end{equation}
so that the computation of this equation is localized to a patch surrounding the node $x\in \mathcal{N}$. We introduce the functions $\xi^l_{k,x} \in V_{\f,k}^x$ as solution to the parabolic equation
\begin{equation}
a(\ddt \xi^l_{k,x}, z) + b(\xi^l_{k,x}, z) = a(\frac{1}{\tau}\chi_{(0,\tau)}(\lambda_x - R_{\f,k}\lambda_x), z), \quad \forall z\in V_{\f,k}^x,
\label{eq:xink}
\end{equation}
with initial value $\xi^0_{k,x} = 0$, and where $\chi_{(0,\tau)}$ is an indicator function on the interval $(0,\tau)$. We claim that $w^n_{k,x} = \sum_{l=1}^n \alpha_x^{n-l}\xi^l_{k,x}$ is the solution to \eqref{eq:wnkx}. This follows as for all $z\in V_{\f,k}^x$
\begin{align*}
	a(\ddt w^n_{k,x}, z&) + b(w^n_{k,x}, z) = a(\ddt \sum_{l=1}^n \alpha_x^{n-l}\xi^l_{k,x}, z) + b(\sum_{l=1}^n \alpha_x^{n-l}\xi^l_{k,x}, z) \\
	&= \sum_{l=2}^n \alpha_x^{n-l} \big( a(\ddt \xi^l_{k,x}, z) + b(\xi^l_{k,x}, z) \big) + \alpha_x^{n-1} \big( a(\ddt \xi^1_{k,x}, z) + b(\xi^1_{k,x}, z) \big) \\
	&= 0 + a(\alpha_x^{n-1}(\lambda_x - R_{\f, k}\lambda_x), z).
\end{align*}
With the localized computations established, the GFEM reads: find $\ulodk{n} = v^n_k + w^n_k$, where $v^n_k = \sum_{x\in \mathcal{N}} \alpha^n_x(\lambda_x - R_{\f,k}\lambda_x) \in V_{\ms,k}$ solves
\begin{equation}
	\tau (\ddt^2v^n_k, z) + a(v^n_k, z) + \tau b(v^n_k, z) = \tau(f^n,z) + a(\ulodk{n-1}, z), \quad \forall z \in V_{\ms, k}, \label{eq_GFEM_full_loc_1}
\end{equation}
and $w^n_k = \sum_{x\in \mathcal{N}} \sum_{l=1}^n \alpha^{n-l}_x \xi^l_{k,x}$, where $\xi^l_{k,x}\in V_{\f,k}^x$ solves \eqref{eq:xink}.

To justify the fact that we localize the finescale equation, we require a result similar to that of Theorem \ref{thm:expdecay}, but for the functions $\{\xi^l_{x}\}_{l=1}^N$. We finish this section about localization by proving that these functions satisfy the exponential decay required for the localization procedure to be viable.

\begin{thm}
	\label{thm:expdecay2}
	For any node $x\in \mathcal{N}$, let $\xi^n_x \in V_\f$ be the solution to 
	\begin{equation*}
	a(\ddt \xi^n_x, z) + b(\xi^n_x, z) = a(\frac{1}{\tau}\chi_{(0,\tau)}(\lambda_x - R_{\f}\lambda_x), z), \quad \forall z\in V_{\f},
	\end{equation*}
	with initial value $\xi^0_{x} = 0$. Then there exist constants $c>0$ and $C>0$ such that for any $k\geq 1$
	\begin{equation*}
	\|\xi^n_x\|_{H^1(\Omega \backslash N^k(x))} \leq 	Ce^{-ck}\|\lambda_x\|_{H^1},
	\end{equation*}
	for sufficiently small time step $\tau$.
\end{thm}

\begin{proof}
	
	First, we analyze the problem for the first time step, which when multiplied by $\tau$ can be written as
	\begin{equation}
	a(\xi^1_x,z) + \tau b(\xi^1_x,z) = a(\lambda_x - \phi_x, z), \quad \forall z\in V_\f,
	\label{eq:test}
	\end{equation}
	where $\phi_x= R_\f\lambda_x$. We denote $\tilde{a} = a+\tau b$ such that $\tilde{a}(\phi_x, z) = \tilde{a}(\lambda_x, z)$ for all $z\in V_\f$. Furthermore we use the energy norm $\vertiii{\cdot} := \sqrt{\tilde{a}(\cdot, \cdot)}$, and by $\vertiii{\cdot}_D$ we denote the restriction of the norm onto a domain $D$. As seen in the proof of Theorem 4.1 in \cite{lodbook}, the result in Theorem \ref{thm:expdecay} can be written as
	\begin{equation*}
	\vertiii{\phi_x}_{\Omega\backslash N^k(x)} \leq C_\phi\mu^{\lfloor k/4 \rfloor}\vertiii{\lambda_x},
	\end{equation*}
	for some $\mu < 1$. Moreover we define the cut-off function $\eta_k \in V_H$ by
	\begin{equation*}
	\eta_k := \begin{cases}
	1, \ \text{in $\Omega\backslash N^{k+1}(x)$}, \\
	0, \ \text{in $N^k(x)$},
	\end{cases}
	\end{equation*}
	for $x\in \mathcal{N}$. Now let $\nu = \eta_{k-3}$. Then we have that
	\begin{align*}
	\supp(\nu) &= \Omega \backslash N^{k-3}(x), \\
	\supp(\nabla \nu) &= N^{k-2}(x)\backslash N^{k-3}(x).
	\end{align*}
	With this setting, we note that
	\begin{align*}
	\vertiii{\xi^1_x}^2_{\Omega\backslash N^k} &\leq \int_\Omega \nu \tilde{A} \nabla \xi^1_x \cdot \nabla \xi^1_x\, \mathrm{d}x = \int_\Omega \tilde{A} \nabla \xi^1_x \cdot \nabla(\nu \xi^1_x)\, \mathrm{d}x - \int_\Omega \tilde{A}\nabla \xi^1_x \cdot \xi^1_x\nabla \nu\, \mathrm{d}x \\
	&\leq \underbrace{\left|\int_\Omega \tilde{A} \nabla \xi^1_x \cdot \nabla(1-I_H)(\nu \xi^1_x)\, \mathrm{d}x \right|}_{=:M_1} + \underbrace{\left|\int_\Omega \tilde{A} \nabla \xi^1_x \cdot \nabla I_H(\nu \xi^1_x)\, \mathrm{d}x \right|}_{=:M_2} \\
	& \qquad \qquad + \underbrace{\left| \int_\Omega \tilde{A}\nabla \xi^1_x \cdot \xi^1_x\nabla \nu\, \mathrm{d}x \right|}_{=:M_3},
	\end{align*}
	where we have denoted $\tilde{A} = A+\tau B$. We now proceed to estimate the terms $M_1$, $M_2$ and $M_3$ separately. For $M_1$, we use the problem \eqref{eq:test} with $z=(1-I_H)(\nu \xi^1_x) \in V_\f$ to get
	\begin{align*}
	M_1 &= \left|\int_\Omega A \nabla (\lambda_x - \phi_x) \cdot \nabla(1-I_H)(\nu \xi^1_x)\, \mathrm{d}x \right| \\ &= \left|\tau \int_{\Omega\backslash N^{k-3}} B \nabla (\phi_x) \cdot \nabla(1-I_H)(\nu \xi^1_x)\, \mathrm{d}x \right|,
	\end{align*}
	where we have used the $\tilde{a}$-orthogonality between $V_\ms$ and $V_\f$, that the integral is zero on $\supp(\lambda_x)$, and that the support of the remaining integrand is $\Omega\backslash N^{k-3}$. Thus, we get that
	\begin{align*}
	M_1 &\leq \tau \frac{\beta_+}{\alpha_-} \vertiii{\phi_x}_{\Omega\backslash N^{k-3}}\vertiii{\xi^1_x}_{\Omega\backslash N^{k-3}} \leq \tau\frac{\beta_+}{\alpha_-} C_\phi \mu^{\lfloor \frac{k-3}{4} \rfloor}\vertiii{\lambda_x}\vertiii{\xi^1_x}_{\Omega\backslash N^{k-4}}.
	\end{align*}
	Moreover, by similar calculations as in the proof of Theorem 4.1 in \cite{lodbook}, from $M_2$ and $M_3$ we get
	\begin{equation*}
	M_2, M_3 \leq \tilde{C}\vertiii{\xi^1_x}^2_{N^{k}\backslash N^{k-4}},
	\end{equation*}
	for a constant $\tilde{C} > 0$. In total, for $\varepsilon \in (0,1)$, we find that
	\begin{align*}
	\vertiii{\xi^1_x}^2_{\Omega\backslash N^k} &\leq \tau\frac{\beta_+}{\alpha_-} C_\phi \mu^{\lfloor \frac{k-3}{4} \rfloor}\vertiii{\lambda_x}\vertiii{\xi^1_x}_{\Omega\backslash N^{k-4}} + \tilde{C}\vertiii{\xi^1_x}^2_{N^{k}\backslash N^{k-4}} \\
	&\leq \frac{(\beta_+ C_\phi)^2}{\alpha^2_{-}\varepsilon}\tau^2\mu^{2\lfloor \frac{k-3}{4} \rfloor}\vertiii{\lambda_x}^2 + \varepsilon\vertiii{\xi^1_x}^2_{\Omega\backslash N^{k-4}} \\
	&\qquad \qquad + \tilde{C}(\vertiii{\xi^1_x}^2_{\Omega\backslash N^{k-4}} - \vertiii{\xi^1_x}^2_{\Omega\backslash N^k}).
	\end{align*}
	Let $\delta := (\varepsilon + \tilde{C})(1+\tilde{C})^{-1} < 1$, and set $\kappa = \max(\delta, \mu) < 1$. Then, by rearranging the terms we get the inequality
	\begin{align*}
	\vertiii{\xi^1_x}^2_{\Omega \backslash N^{k}} \leq \frac{(\beta_+ C_\phi)^2}{\alpha^2_{-}\varepsilon(1+\tilde{C})}\tau^2\kappa^{2\lfloor \frac{k-3}{4} \rfloor}\vertiii{\lambda_x}^2 + \kappa \vertiii{\xi^1_x}^2_{\Omega \backslash N^{k-4}}.
	\end{align*}
	Repeating the estimate, we end up with
	\begin{align*}
	\vertiii{\xi^1_x}^2_{\Omega \backslash N^{k}} \leq \kappa^{\lfloor k/4 \rfloor}\vertiii{\xi^1_x}^2_{\Omega} + \frac{(\beta_+ C_\phi)^2}{\alpha^2_{-}\varepsilon(1+\tilde{C})}\vertiii{\lambda_x}^2\sum_{i=0}^{\lfloor k/4 \rfloor-1} \tau^2\kappa^{i}\kappa^{2\lfloor \frac{k-3-4i}{4} \rfloor}.
	\end{align*}
	We proceed by estimating $\vertiii{\xi^1_x}_\Omega$. By choosing $z = \xi^1_x$ in \eqref{eq:test} we get
	\begin{align*}
	\vertiii{\xi^1_x}^2 \leq \vertiii{\lambda_x - \phi_x}\vertiii{\xi^1_x} \leq \vertiii{\lambda_x} \vertiii{\xi^1_x},
	\end{align*}
	since
	\begin{align*}
	\vertiii{\lambda_x - \phi_x}^2 &= \tilde{a}(\lambda_x-\phi_x, \lambda_x-\phi_x) \leq \vertiii{\lambda_x -\phi_x} \vertiii{\lambda_x}.
	\end{align*}
	Moreover, for $i=0,1,2,...,\lfloor k/4 \rfloor-1$, we note that
	\begin{align}
	\kappa^{i+2\lfloor \frac{k-3-4i}{4} \rfloor} \leq  \kappa^{\lfloor k/4\rfloor-1+2\lfloor \frac{k-3-4(k/4-1)}{4} \rfloor} = \kappa^{\lfloor k/4\rfloor-1+2\lfloor \frac{1}{4} \rfloor} = \kappa^{\lfloor k/4\rfloor-1}
	\label{eq:sumest}
	\end{align}
	so in total we have the estimate
	\begin{equation*}
	\vertiii{\xi^1_x}_{\Omega\backslash N^{k}} \leq \sqrt{1+C_0\tau^2}\kappa^{\frac{1}{2}\lfloor k/4\rfloor}\vertiii{\lambda_x}, \quad \text{ with } \quad C_0 = \frac{(\beta_+ C_\phi)^2\kappa^{-1}}{\alpha^2_{-}\varepsilon(1+\tilde{C})}(\lfloor k/4 \rfloor - 1).
	\end{equation*}
	Recall that this is for the first time step. In next time step, we consider the problem 
	\begin{equation*}
	a(\xi^2_x,z) + \tau b(\xi^2_x,z) = a(\xi^1_x, z), \quad \forall z\in V_\f.
	\end{equation*}
	As for the first time step, we split the estimate into the similar integrals $M_1$, $M_2$, and $M_3$, and get
	\begin{equation*}
	M_1 \leq \tau\frac{\beta_+}{\alpha_-}\vertiii{\xi^1_x}_{\Omega\backslash N^{k-3}}\vertiii{\xi^2_x}_{\Omega \backslash N^{k-3}} \leq \tau \frac{\beta_+}{\alpha_-} \sqrt{1+C_0\tau^2}\kappa^{\frac{1}{2}\lfloor \frac{k-3}{4} \rfloor}\vertiii{\lambda_x} \vertiii{\xi^2_x}_{\Omega \backslash N^{k-4}},
	\end{equation*}
	while $M_2$ and $M_3$ remain the same. In total, we get the estimate
	\begin{align*}
	(1+\tilde{C})\vertiii{\xi^2_x}^2_{\Omega \backslash N^k} \leq \frac{\beta_+^2}{\alpha^2_{-}\varepsilon}\tau^2(1+C_0\tau^2)\kappa^{\lfloor \frac{k-3}{4} \rfloor} \vertiii{\lambda_x}^2 + (\varepsilon + \tilde{C})\vertiii{\xi^2_x}^2_{\Omega \backslash N^{k-4}}.
	\end{align*}
	Once again, by letting $\delta = (\varepsilon + \tilde{C})/(1+\tilde{C})$ and since $\delta \leq \kappa$, we get
	\begin{align*}
	\vertiii{\xi^2_x}^2_{\Omega \backslash N^k} &\leq \frac{\beta^2_+}{\alpha^2_{-}\varepsilon(1+\tilde{C})}\tau^2(1+C_0\tau^2)\kappa^{\lfloor \frac{k-3}{4} \rfloor} \vertiii{\lambda_x}^2 + \kappa \vertiii{\xi^2_x}^2_{\Omega \backslash N^{k-4}} \\
	&\leq \kappa^{\lfloor k/4 \rfloor} \vertiii{\lambda_x}^2 +\frac{\beta^2_+}{\alpha^2_{-}\varepsilon(1+\tilde{C})}(1+C_0\tau^2)\sum_{i=0}^{\lfloor k/4\rfloor - 1} \tau^2 \kappa^{i} \kappa^{\lfloor \frac{k-3-4i}{4} \rfloor}.
	\end{align*}
	Once again we use \eqref{eq:sumest} to conclude that
	\begin{align*}
	\vertiii{\xi^2_x}_{\Omega \backslash N^k} &\leq \sqrt{1+C_1\tau^2(1+C_0\tau^2)}\kappa^{\frac{1}{2}\lfloor k/4 \rfloor}\vertiii{\lambda_x} \\ &=\sqrt{1+C_1\tau^2 + C_1\tau^2 C_0\tau^2}\kappa^{\frac{1}{2}\lfloor k/4 \rfloor}\vertiii{\lambda_x},
	\end{align*}
	where 
	\begin{equation*}
	C_1 = \frac{\beta^2_+\kappa^{-1}}{\alpha^2_{-}\varepsilon(1+\tilde{C})}(\lfloor k/4\rfloor - 1).
	\end{equation*}
	Inductively, we get for arbitrary time step $n$ the estimate
	\begin{equation*}
	\vertiii{\xi^n_x}_{\Omega \backslash N^k} \leq \kappa^{\frac{1}{2}\lfloor k/4 \rfloor}\vertiii{\lambda_x} \sqrt{\sum_{i=0}^{n-1}(C_1\tau^2)^i +(C_1\tau^2)^nC_0\tau^2}.
	\end{equation*}
	Since $\kappa^{\frac{1}{2}\lfloor k/4 \rfloor} \leq  Ce^{-ck}$ for some $c > 0$ and $C>0$, and since the energy norm is equivalent to the $H^1$-norm, the theorem holds.
	
\end{proof}

\begin{remark}
	Note that the constant that appears in the final inequality converges to 
	\begin{equation*}
	\frac{1}{1-C_1\tau^2},
	\end{equation*}
	which means that the constant behaves nicely for sufficiently small time steps. More specifically, for time steps
	\begin{equation*}
	\tau \leq \sqrt{\frac{\alpha^2_{-}\varepsilon\kappa(1+\tilde{C})}{\beta^2_+(\lfloor k/4\rfloor - 1)}}.
	\end{equation*}
\end{remark}

%

\section{Error estimates}
\label{sec:errorestimates}

In this section we derive error estimates of the ideal method \eqref{eq_GFEM_full_1}-\eqref{eq:GFEM_full_2}. The additional error due to localization can be controlled in terms of the localization parameter $k$. This is further discussed in Remark \ref{rmk:loc}. We begin by considering an auxiliary problem.

\subsection{Auxiliary problem}

The auxiliary problem is defined as the standard variational formulation for the strongly damped wave equation, but we exclude the second order time derivative. Moreover, we let the starting time $t = t_0$ be general and set the time discretization to $t = t_0 < t_1 < ...< t_N = T$. Thus, the auxiliary problem is to find $\Z{n}\in V_h$ for $n=1,...,N$, such that
\begin{equation}
a(\ddt \Z{n}, v) + b(\Z{n}, v) = (f^n,v), \quad \forall v\in V_h,
\label{fullydisc_simple}
\end{equation}
with initial value $\Z{0} \in V_\ms$. Equivalently, multiply \eqref{fullydisc_simple} by $\tau$ and we may consider
\begin{equation}
a(\Z{n}, v) + \tau b(\Z{n}, v) = \tau (f^n,v) + a(\Z{n-1}, v), \quad \forall v\in V_h.
\label{fullydisc_simple2}
\end{equation}
Existence of a solution to this problem is guaranteed by Lax--Milgram. For simplicity, we make the assumption that the initial data for the damped wave equation \eqref{eq:fullydiscrete} is already in the multiscale space $V_\ms$, such that $$u^0_h=\ulod{0} \in V_\ms, \quad \, u^1_h=\ulod{1} \in V_\ms.$$ For general initial data we refer to Section \ref{sec:initialdata} below. Furthermore, to limit the technical details in the proof we have chosen to analyze the error in the $L_2(H^1)$-norm instead of the pointwise (in time) $H^1$-norm.

The solution space can be decomposed as $V_h = V_\ms \oplus V_\f$, such that the solution can be written as $\Z{n} = v^n + w^n$ where $v^n \in V_\ms$ and $w^n \in V_\f$. If we insert this into the system in \eqref{fullydisc_simple2} and consider test functions $z\in V_\ms$, the left hand side becomes
\begin{align*}
a(\Z{n},z) + \tau b(\Z{n}, z) 
&= a(v^n, z) + \tau b(v^n, z),
\end{align*}
where we have used the orthogonality between $V_\ms$ and $V_\f$ with respect to $a(\cdot, \cdot) + \tau b(\cdot, \cdot)$. Likewise, if test functions $z\in V_\f$ are considered, the left hand side becomes
\begin{align*}
a(\Z{n}, z) + \tau b(\Z{n}, z) = a(w^n, z) + \tau b(w^n, z).
\end{align*}
With these findings, we define the approximation to the auxiliary problem as to find $\Zlod{n} = v^n + w^n$, where $v^n \in V_\ms$ and $w^n \in V_\f$ such that
\begin{alignat}{2}
a(v^n, z) + \tau b(v^n, z) &= \tau(f^n,z) + a(\Zlod{n-1}, z),& \quad &\forall z \in V_\ms,\label{GFEM_simple_1} \\
a(w^n, z) + \tau b(w^n, z) &= a(\Zlod{n-1}, z), & & \forall z \in V_\f,\label{GFEM_simple_2}
\end{alignat}
with initial data $\Zlod{0} \in V_\ms$. Note that if $f=0$, then $\Z{n} = \Zlod{n}$ for every $n$, meaning that the method reproduces $\Z{n}$ exactly. For the auxiliary problem, we prove the following error estimates.

\begin{thm}\label{simple_problem_err}
	Let $\Z{n}$ be the solution to \eqref{fullydisc_simple} and $\Zlod{n}$ the solution to \eqref{GFEM_simple_1}-\eqref{GFEM_simple_2}. Assume that $\Zlod{0} - \Z{0} = 0$, then the error is bounded by
	\begin{align} 
	\|\Z{n} - \Zlod{n}\|_{H^1} &\leq CH\sum_{j=1}^n\tau\|f^{j}\|_{L_2}. \label{simple_problem_err1}
	\end{align}
	If $f^n \in L_2(\Omega)$, for $n \geq 0$, then we have
	\begin{align}
	\sum_{j=1}^n \tau \|\Z{j} - \Zlod{j}\|^2_{L_2} &\leq CH^2\sum_{j=1}^n \tau\|f^j\|^2_{L_2}  , \label{simple_problem_err2}
	\end{align}
	and if $f^n=\ddt g^n$, for some $\{g^n\}_{n=0}^N$ such that  $g^n \in V_h$, then
	\begin{align}
	\sum_{j=1}^n \tau \|\Z{j} - \Zlod{j}\|^2_{L_2} &\leq CH^2\Bigg(\sum_{j=1}^n \tau \|g^j\|^2_{L_2} + \|g^0\|^2_{L_2}\Bigg), \label{simple_problem_err4}
	\end{align}
	where C does not depend on the variations in $A$ or $B$.
\end{thm}

\begin{proof}
	Since $\Z{n} \in V_h$ there are $\bar v^n \in V_\ms$ and $\bar w^n \in V_\f$ such that $\Z{n} = \bar v^n + \bar w^n$. Let $e^n = \Z{n} - \Zlod{n}$, and consider
	\begin{align*}
	\vertiii{e^n}^2 &:=a(e^n,e^n) + \tau b(e^n,e^n) \\
	&= \tau(f^n,e^n) + a(\Z{n-1},e^n)- a(v^n,e^n) - \tau b(v^n, e^n) - a(w^n,e^n) - \tau b(w^n,e^n).
	\end{align*}
	For $v^n \in V_\ms$ we have due to the orthogonality and \eqref{GFEM_simple_1}
	\begin{align*}
	a(v^n,e^n) + \tau b(v^n, e^n) &= a(v^n,\bar v^n - v^n) + \tau b(v^n, \bar v^n - v^n) \\&= \tau(f^n,\bar v^n  - v^n) + a(\Zlod{n-1},\bar v^n - v^n).
	\end{align*}
	Similarly, for $w^n \in V_\f$ we use the orthogonality and \eqref{GFEM_simple_2} to get
	\begin{align*}
	a(w^n,e^n) + \tau b(w^n, e^n) &
	= a(\Zlod{n-1}, \bar w^n - w^n).
	\end{align*}
	Hence,
	\begin{align*}
	\vertiii{e^n}^2 &= \tau(f^n,e^n) + a(\Z{n-1},e^n)- \tau(f^n,\bar v^n  - v^n) \\ &\qquad \qquad - a(\Zlod{n-1},\bar v^n - v^n) - a(\Zlod{n-1}, \bar w^n - w^n)\\
	&= \tau(f^n,\bar w^n - w^n) + a(\Z{n-1} - \Zlod{n-1},e^n).
	\end{align*}
	The first term can be bounded by using the interpolation operator $I_H$
	\begin{align*}
	\tau|(f^n,\bar w^n - w^n)| &\leq \tau\|f^n\|_{L_2}\|\bar w^n - w^n - I_H(\bar w^n - w^n)\|_{L_2} \\ &\leq CH\tau\|f^n\|_{L_2}\|\bar w^n - w^n\|_{H^1} \\& \leq  CH\tau\|f^n\|_{L_2}\|e^n\|_{H^1}  \\ &\leq  CH\tau\|f^n\|_{L_2}\vertiii{e^n}.
	\end{align*}
	For the second term we note that $\Z{n-1} - \Zlod{n-1} = e^{n-1}$ so that
	\begin{align*}
	\vertiii{e^n} &\leq CH\tau\|f^n\|_{L_2} + \vertiii{e^{n-1}}. 
	\end{align*}
	Using this bound repeatedly and $e^0=0$ we get
	\begin{align*}
	\vertiii{e^n} &\leq CH\sum_{j=1}^n\tau\|f^{j}\|_{L_2}. 
	\end{align*}
	This concludes the proof since $\|e^n\|_{H^1} \leq C\vertiii{e^n}$.
	
	To prove the remaining bounds in $L_2$-norm, we define the forward difference operator $\tilde \partial_t x^n = (x^{n+1}-x^n)/\tau$ and consider the dual problem: find $x^j_h \in V_h$ for $j=n-1,...,0$, such that $x^n_h = 0$ and
	\begin{align}\label{dual_problem}
	a(-\tilde \partial_t x^j_h,z) + b(x^j_h,z) = (e^{j+1},z), \quad \forall z \in V_h.
	\end{align}
	Note that this problem moves backwards in time. By choosing $z = x^j_h$ in \eqref{dual_problem} and performing a classical energy argument, we deduce
	\begin{align}\label{dual_energy}
	\|x^j_h\|^2_{H^1} + \sum_{k=j}^n\tau \|x^k_h\|^2_{H^1} \leq C\sum_{k=j+1}^n\tau\|e^k\|^2_{L_2}.
	\end{align}
	Similarly, by choosing $z=-\tilde \partial_t x^j_h$, we achieve
	\begin{align}\label{dual_energy_timeder}
	\|x^j_h\|^2_{H^1} + \sum_{k=j}^n\tau \|\tilde \partial_t x^k_h\|^2_{H^1} \leq C\sum_{k=j+1}^n\tau\|e^k\|^2_{L_2}.
	\end{align}
	Now, use  \eqref{dual_problem} to get
	\begin{align*}
	\sum_{j=1}^n \tau \|e^{j}\|^2_{L_2} = \sum_{j=1}^n \tau a(-\tilde \partial_t x^{j-1}_h,e^j) + \tau b(x^{j-1}_h,e^j).
	\end{align*}
	Summation by parts gives 
	\begin{align}
	\sum_{j=1}^n \tau \|e^{j}\|^2_{L_2} &= \sum_{j=1}^n \tau a(-\tilde \partial_t x^{j-1}_h,e^j) + \tau b(x^{j-1}_h,e^j) \label{dual_problem_sum}\\&= \sum_{j=1}^n \tau a(x^{j-1}_h,\ddt e^j) + \tau b(x^{j-1}_h,e^j),\notag
	\end{align}
	where we have used $x^n=e^0=0$. Furthermore, we use the equations \eqref{fullydisc_simple} and \eqref{GFEM_simple_1}, and the orthogonality in \eqref{MSspace}, to show that the following Galerkin orthogonality holds for $z_\ms \in V_\ms$
	\begin{align}\label{Galerkin_Orthogonality_simple}
	a(\ddt e^j, z_\ms) + b(e^j, z_\ms) &= a(\ddt \Z{j}, z_\ms) + b(\Z{j}, z_\ms) - \frac{1}{\tau}a(v^j, z_\ms) \\&\qquad \qquad - b(v^j, z_\ms) + \frac{1}{\tau}a(\Zlod{j-1},z_\ms)  \\&= (f^j,z_\ms) - (f^j,z_\ms) = 0.\notag
	\end{align}
	Let $x^j_h = x_\ms + x_\f$, for some $x_\ms \in V_\ms,\,  x_\f \in V_\f$. Using the orthogonality \eqref{Galerkin_Orthogonality_simple} and the equations \eqref{GFEM_simple_2} and \eqref{fullydisc_simple} we deduce
	\begin{align*}
	\sum_{j=1}^n \tau a(x^{j-1}_h,\ddt e^j) + \tau b(x^{j-1}_h,e^j) &= \sum_{j=1}^n \tau a(x^{j-1}_\f,\ddt e^j) + \tau b(x^{j-1}_\f,e^j)\\&= \sum_{j=1}^n \tau a(x^{j-1}_\f,\ddt \Z{j}) + \tau b(x^{j-1}_\f,\Z{j})\\
	&= \sum_{j=1}^n \tau (x^{j-1}_\f,f^j).
	\end{align*}
	If $f^j \in L_2(\Omega)$, then we may subtract $I_H x_\f = 0$ and use \eqref{eq:intestimate} to achieve
	\begin{align*}
	\sum_{j=1}^n \tau (x^{j-1}_\f,f^j) &\leq CH \sum_{j=1}^n \tau \|x^{j-1}_\f\|_{H^1}\|f^j\|_{L_2} \\&\leq  CH \Bigg(\sum_{j=1}^n \tau \|x^{j-1}_\f\|^2_{H^1}\Bigg)^{1/2}\Bigg(\sum_{j=1}^n \tau \|f^j\|^2_{L_2}\Bigg)^{1/2}.
	\end{align*}
	Note that $\vertiii{x^{j-1}_\f}^2 + \vertiii{x^{j-1}_\ms}^2 \leq \vertiii{x^{j-1}_h}^2$. Hence the energy estimate \eqref{dual_energy} can now be used to achieve \eqref{simple_problem_err2}.	
	If $f^j = \ddt g^j$ one may use summation by parts to achieve
	\begin{align*}
	\sum_{j=1}^n \tau \|e^{j}\|^2_{L_2} &= \sum_{j=1}^n \tau (x^{j-1}_\f,\ddt g^j) \leq \sum_{j=1}^n \tau (- \tilde \partial_t x^{j-1}_\f,g^j) - (x^0_\f,g^0) \\& \leq CH\sum_{j=1}^n \tau \|\tilde \partial_t x^{j-1}_\f\|_{H^1}\|g^j\|_{L_2} + CH\|x^0_\f\|_{H^1}\|g^0\|_{L_2},
	\end{align*}
	where we have used $x^n_\f=x^n_h=0$. Using \eqref{dual_energy_timeder} we conclude \eqref{simple_problem_err4}.
\end{proof}

\begin{remark}
	The bound in \eqref{simple_problem_err2} is not of optimal order, but it is useful in the error analysis.
\end{remark}

The next lemma gives error estimates for the discrete time derivative of the error. In the analysis of the (full) damped wave equation we use $g = \ddt u_h$, see Lemma~\ref{RitzVolterra_err_timederivative}. If the initial data is nonzero we expect $\ddt g^1$ below to be of order $t^{-1}$ in $L_2$-norm. A detailed explanation of this is given below. Hence, we have a blow up close to zero due to low regularity of the initial data. Therefore, we need to multiply the error by $t_j$. This is similar to the parabolic case for nonsmooth initial data see, e.g., \cite{Thomee06}.  


\begin{lemma}\label{simple_problem_err_timederivative}
	Let $\Z{n}$ be the solution to \eqref{fullydisc_simple} and $\Zlod{n}$ the solution to \eqref{GFEM_simple_1}-\eqref{GFEM_simple_2}. 
	Assume $\Zlod{0} - \Z{0} =0$.	If $\ddt f^n \in L_2(\Omega)$, for $n \geq 1$, then
	\begin{align}
	\label{simple_problem_err_timederivative1}\sum_{j=2}^n \tau \|\ddt (\Z{j} - \Zlod{j})\|^2_{L_2} &\leq CH^2\Bigg(\sum_{j=2}^n \tau \|\ddt f^j\|^2_{L_2} + \|f^1\|^2_{L_2}\Bigg)
	\end{align}
	and if $f^n=\ddt g^n$, for some $\{g^n\}_{n=0}^N$, such that  $g^n \in V_h$, then
	\begin{align}
	\label{simple_problem_err_timederivative2}\sum_{j=2}^n \tau \|\ddt (\Z{j} - \Zlod{j})\|^2_{L_2}\ &\leq CH^2\Bigg(\sum_{j=2}^n \tau \|\ddt g^j\|^2_{L_2} + \|\ddt g^1\|^2_{L_2}\Bigg)
	\end{align}
	and, in addition, the following bound holds
	\begin{align}
	\label{simple_problem_err_timederivative3}\sum_{j=2}^n &\tau t_j^2\|\ddt (\Z{j} - \Zlod{j})\|^2_{L_2} \leq CH^2\Bigg(\sum_{j=2}^n \tau  \|\ddt g^j\|^2_{L_2} + t_1^2\|\ddt g^1\|^2_{L_2}\Bigg),
	\end{align}
	where C does not depend on the variations in $A$ or $B$.
\end{lemma}

\begin{proof}
	The proof of \eqref{simple_problem_err_timederivative1} is similar to \eqref{simple_problem_err2}. Let $e^j = \Z{j} - \Zlod{j}$ and define the dual problem
	\begin{align}\label{dual_problem_timederivative}
	a(-\tilde \partial_t x^j_h,z) + b(x^j_h,z) = (\ddt e^{j+1},z), \quad \forall z \in V_h, \ j = n-1,...,0,
	\end{align}
	with $x^n = 0$. Choosing $z=\ddt e^{j+1}$ and performing summation by parts we deduce
	\begin{align}\label{simple_error_eq2}
	\sum_{j=2}^n \tau \|\ddt e^{j}\|^2_{L_2} &= \sum_{j=2}^n \tau a(-\tilde \partial_t x^{j-1}, \ddt e^j) + \tau b(x^{j-1}_h,\ddt e^j) \\&=\sum_{j=2}^n \tau a(x^{j-1}_h,\ddt^2 e^j) + \tau b(x^{j-1}_h,\ddt e^j) + a(x^1_h,\ddt e^1),\notag
	\end{align}
	where we used that $x^n = 0$. Following the same argument as for \eqref{simple_problem_err4}, but with a difference quotient, we arrive at
	\begin{align*}
	\sum_{j=2}^n \|\ddt e^j\|^2_{L_2}&= \sum_{j=2}^n \tau a(x^{j-1}_h,\ddt^2 e^j) + \tau b(x^{j-1}_h,\ddt e^j) + a(x^1_h,\ddt e^1) 
	\\&= \sum_{j=2}^n \tau (x^{j-1}_\f,\ddt f^j) + a(x^1_h,\ddt e^1).
	\end{align*}
	Using $e^0=0$, we deduce
	\begin{align*}
	a(x^1_h,\ddt e^1) = \frac{1}{\tau}a(x^1_h,e^1) \leq \frac{C}{\tau}H\|x^1_h\|_{H^1}\tau \|f^1\|_{L_2} \leq CH\|x^1_h\|_{H^1} \|f^1\|_{L_2},
	\end{align*}
	and with $\ddt f^j \in L_2(\Omega)$ we get
	\begin{align*}
	\sum_{j=2}^n \|\ddt e^j\|^2_{L_2}	&\leq CH\sum_{j=2}^n \tau \|x^{j-1}_\f\|_{H^1}\|\ddt f^j\|_{L_2} + CH\|x^1_h\|_{H^1} \|f^1\|_{L_2},
	\end{align*}
	and \eqref{simple_problem_err_timederivative1} follows by using an energy estimate of $x^{j}_h$ similar to \eqref{dual_energy}, but with $\ddt e^j$ on the right hand side.
	
	If $f^j =\ddt g^j$ we proceed as for \eqref{simple_problem_err4} to achieve 
	\begin{align*}
	\sum_{j=2}^n \tau (x^{j-1}_\f,\ddt f^j) \leq CH\sum_{j=2}^n \tau \|\tilde \partial_t x^{j-1}_\f\|_{H^1}\|\ddt g^j\|_{L_2} + CH\|x^1_\f\|_{H^1}\|\ddt g^1\|_{L_2}
	\end{align*}
	and \eqref{simple_problem_err_timederivative2} follows by using energy estimates similar to \eqref{dual_energy_timeder}.
	
	For \eqref{simple_problem_err_timederivative3} we consider the dual problem
	\begin{align}\label{dual_problem_t}
	a(-\tilde \partial_t x^j_h,z) + b(x^j_h,z) = (t_{j+1}\ddt e^{j+1},z), \quad \forall z \in V_h, \ j = n-1,...,0.
	\end{align}
	A simple energy estimate shows
	\begin{align}\label{dual_problem_est}
	\|x^j_h\|^2_{H^1} + \sum_{k=j}^{n-1}\tau\|\tilde\partial_t x^k_h\|^2_{H^1} \leq C\sum_{k=j}^{n-1}\tau t^2_{k+1}\|\ddt e^{k+1}\|^2_{L_2}, \quad j=0,...,n-1.
	\end{align}
	Now choosing $z=t_{j+1}\ddt e^{j+1}$ in \eqref{dual_problem_t} and performing summation by parts gives
	\begin{align}\label{simple_error_eq3}
	\sum_{j=2}^n \tau t^2_j\|\ddt e^{j}\|^2_{L_2} &= \sum_{j=2}^n \tau a(-\tilde \partial_t x^{j-1}_h, t_j\ddt e^j) + \tau b(x^{j-1}_h,t_j\ddt e^j) \\&= \sum_{j=2}^n \big(\tau a(x^{j-1}_h,t_j\ddt^2 e^j) + \tau b(x^{j-1}_h,t_j\ddt e^j) \notag\\&\quad+ a(x^{j-1}_h, (t_j - t_{j-1})\ddt e^{j-1})\big)  + a(x^1_h,t_1\ddt e^1).\notag
	\end{align}
	The first two terms of the sum can be handled similarly to \eqref{simple_problem_err_timederivative2},
	\begin{align*}
	\sum_{j=2}^n \tau a(x^{j-1}_h,t_j\ddt^2 e^j) + \tau b(x^{j-1}_h,t_j\ddt e^j) = \sum_{j=2}^n \tau (x^{j-1}_\f, t_j\ddt f^j).
	\end{align*}
	Now, using summation by parts we achieve
	\begin{align*}
	\sum_{j=2}^n &\tau (x^{j-1}_\f, t_j\ddt f^j) = \sum_{j=2}^n \tau \big((-\tilde \partial_t x^{j-1}_\f, t_jf^j) + ( x^{j-1}_\f, f^j)\big) + (x^1_\f,t_1 f^1) \\ &\leq CH\Bigg( \sum_{j=2}^n \tau \big(t_j\|\tilde \partial_t x^{j-1}_\f\|_{H^1}\|f^j\|_{L_2} + \|x^{j-1}_\f\|_{H^1} \|f^j\|_{L_2} \big) + t_1\|x^1_\f\|_{H^1}\|f^1\|_{L_2}\Bigg)
	\end{align*}
	where we can use \eqref{dual_problem_est}. Note that in the first term we can use the (crude) bound  $t_j^2\leq t_n^2$ and let the constant $C$ depend on $T$. We get
	\begin{align*}
	\sum_{j=2}^n \tau (x^{j-1}_\f, t_j\ddt f^j) &\leq CH\Bigg( \sum_{j=1}^{n}\tau t^2_j\|\ddt e^{j}\|^2_{L_2}\Bigg)^{1/2} \Bigg(\Bigg( \sum_{j=2}^n \tau  \|f^j\|^2_{L_2} \Bigg)^{1/2} + t_1\|f^1\|_{L_2}\Bigg).
	\end{align*}
	For the third term in \eqref{simple_error_eq3}, we use $t_j-t_{j-1}=\tau$ and once again perform summation by parts to get
	\begin{align*}
	\sum_{j=2}^n \tau a(x^{j-1}_h,\ddt e^{j-1}) = \sum_{j=2}^n \tau a(- \tilde\partial_t x^{j-1}_h, e^{j-1}),
	\end{align*}
	where we have used $x^n_h=e^0=0$. Combining \eqref{dual_problem_est} and \eqref{simple_problem_err1} we get
	\begin{align*}
	\sum_{j=2}^n \tau a(- \tilde\partial_t x^{j-1}_h, e^{j-1}) &\leq C \max_{j=1,...,n}\|e^j\|_{H^1} \Bigg(\sum_{j=2}^n \tau \Bigg)^{1/2}\Bigg(\sum_{j=2}^n \tau \|\tilde\partial_t x^{j-1}_h\|^2_{H^1} \Bigg)^{1/2} \\& \leq CH \sum_{j=1}^n \tau\|f^j\|_{L_2} \Bigg(\sum_{j=1}^{n}\tau t_j^2\|\ddt e^{j}\|^2_{L_2}\Bigg)^{1/2}.
	\end{align*}
	For the last term in \eqref{simple_error_eq3} we use \eqref{dual_problem_est} and \eqref{simple_problem_err1} for $n=1$ to achieve
	\begin{align*}
	a(x^1_h,t_1\ddt e^1) = a(x^1_h,e^1) \leq CH\Bigg( \sum_{j=1}^{n}\tau t^2_j\|\ddt e^{j}\|^2_{L_2}\Bigg)^{1/2}t_1\|f^1\|_{L_2},
	\end{align*}
	and \eqref{simple_problem_err_timederivative3} follows by letting $f^j = \ddt g^j$.
	
	
\end{proof}

\subsection{The damped wave equation}

For the error analysis of the full damped wave equation we shall make use of the projection corresponding to the auxiliary problem. For $u^n_h \in V_h$, let $X^n = X^n_v + X^n_w \in V_h$ with $X^n_v\in V_\ms$ and $X^n_w\in V_\f$ such that
\begin{alignat}{2}
a(X^n_v - u^n_h,z) + \tau b(X^n_v - u^n_h, z) &= a(X^{n-1} - u^{n-1}_h, z), \quad &&\forall z\in V_\ms, \label{eq:Xproj1} \\
a(X^n_w, z) + \tau b(X^n_w, z) &= a(X^{n-1}, z), \quad &&\forall z\in V_\f.\label{eq:Xproj2}
\end{alignat}
Note that since $u^n_h$ solves \eqref{eq:fullydiscrete}, the system \eqref{eq:Xproj1}-\eqref{eq:Xproj2} is equivalent to
\begin{alignat}{2}
a(X^n_v,z) + \tau b(X^n_v, z) &= \tau(f^n - \ddt^2u^n_h, z) + a(X^{n-1}, z), \quad &&\forall z\in V_\ms, \\
a(X^n_w, z) + \tau b(X^n_w, z) &= a(X^{n-1}, z), \quad &&\forall z\in V_\f.
\end{alignat}
That is, we may view $u^n_h$ and $X^n$ as the solution and approximation to the auxiliary problem with source data $f^n-\ddt^2u^n_h$. We deduce following lemma. 

\begin{lemma}\label{RitzVolterra_err}
	Let $u^n_h$ be the solution to \eqref{eq:fullydiscrete} and $X^n$ the solution to \eqref{eq:Xproj1}-\eqref{eq:Xproj2}. The error satisfies the following bounds
	\begin{align}
	\|X^n - u^n_h\|_{H^1} &\leq CH\sum_{j=2}^n\tau\|f^{j} - \ddt^2 u^j_h\|_{L_2}, \quad n\geq 2,\label{RitzVolterra_err_1}\\
	\sum_{j=2}^n \tau \|X^n - u^n_h\|^2_{L_2} &\leq CH^2\Bigg(\sum_{j=2}^n \tau (\|f^j\|^2_{L_2} + \|\ddt u^j_h\|^2_{L_2}) + \|\ddt u^1_h\|^2_{L_2}\Bigg), \quad n\geq 2,  \label{RitzVolterra_err_2}
	\end{align}
	where C does not depend on the variations in $A$ or $B$.
\end{lemma}

\begin{proof}
	We let the auxiliary problem \eqref{fullydisc_simple} start at $t_1$. In this case $e^0 = u^1_h - u^1_\ms = 0$, since $u^1_\ms=u^1_h \in V_\ms$. The bound \eqref{RitzVolterra_err_1} now follows directly from \eqref{simple_problem_err1} with $f^n-\ddt^2u^n_h$ as right hand side. The second bound \eqref{RitzVolterra_err_2} follows from \eqref{simple_problem_err2} and \eqref{simple_problem_err4} with $f^n \in L_2(\Omega)$ and $g^n = \ddt u^{n+1}_h$.
\end{proof}

In a similar way me may deduce bounds for the (discrete) time derivative of the error. As a direct consequence of Lemma~\ref{simple_problem_err_timederivative}, we get the following result.

\begin{lemma}\label{RitzVolterra_err_timederivative}
	Let $u^n_h$ be the solution to \eqref{eq:fullydiscrete} and $X^n$ the solution to \eqref{eq:Xproj1}-\eqref{eq:Xproj2}. The following bounds hold
	\begin{align}
	\sum_{j=3}^n \tau \|\ddt &(X^j - u^j_h)\|^2_{L_2} \\&\leq CH^2\Bigg(\sum_{j=3}^n \tau (\|\ddt f^j\|^2_{L_2} +\|\ddt^2 u^j_h\|^2_{L_2}) +\|f^2\|^2_{L_2} + \|\ddt^2 u^2_h\|^2_{L_2}\Bigg),\notag\\
	\sum_{j=3}^n \tau t_j^2\|&\ddt (X^j - u^j_h)\|^2_{L_2} \\&\leq CH^2\Bigg(\sum_{j=3}^n \tau (\|\ddt f^j\|^2_{L_2} +\|\ddt^2 u^j_h\|^2_{L_2}) +t_2^2\|f^2\|^2_{L_2} + t_2^2\|\ddt^2 u^2_h\|^2_{L_2}\Bigg),\notag
	\end{align}
	where C does not depend on the variations in $A$ or $B$. 
	%
\end{lemma}


\begin{lemma}\label{lemma_error_zero_initial_data}
	Let $u^n_h$ and $\ulod{n}$ be the solutions to \eqref{eq:fullydiscrete} and \eqref{eq_GFEM_full_1}-\eqref{eq:GFEM_full_2}, respectively. Assume that $u_0 = u_1 =0$. The error is bounded by
	\begin{align*}
	\sum_{j=2}^n\tau\|\ulod{j} - u^j_h\|^2_{H^1} \leq C H^2 \Bigg( \sum_{j=1}^n \tau (\|f^j\|^2_{L_2} + \|\ddt f^j\|^2_{L_2}) + \max_{j=1,...,n} \|f^j\|^2_{L_2}\Bigg),
	\end{align*}
	for $n \geq 2$, where C does not depend on the variations in $A$ or $B$.
\end{lemma}
\begin{proof}
	Begin by splitting the error into two contributions
	\begin{align*}
	\ulod{n} - u^n_h = \ulod{n} - X^n + X^n - u^n_h=: \theta^n + \rho^n,
	\end{align*}
	where $X^n$ is the solution to the simplified problem in\eqref{eq:Xproj1}-\eqref{eq:Xproj2}. By Lemma~\ref{RitzVolterra_err} $\rho^n$ is bounded by
	\begin{align*}
	\|\rho^n\|_{H^1} \leq CH\sum_{j=2}^n\tau\big(\|f^{j}\|_{L_2} + \|\ddt^2 u^j_h\|_{L_2}\big),
	\end{align*}
	and we can now apply the energy bound \eqref{energy_bound2}. It remains to bound $\theta^n$. Recall that for any $z \in V_h$ we have $z = z_\ms + z_\f$ for some $z_\ms \in V_\ms$ and $z_\f \in V_\f$. Using that $\ulod{n}=v^n + w^n$ satisfies \eqref{eq_GFEM_full_1} and the orthogonality \eqref{MSspace} we get
	\begin{align*}
	(\ddt^2 \ulod{n}, z_\ms) + a(\ddt \ulod{n}, z_\ms) + b(\ulod{n}, z_\ms) &= (f^n,z_\ms) + (\ddt^2 w^n,z_\ms).
	\end{align*}
	Similarly, due to \eqref{eq:GFEM_full_2} and the orthogonality, 
	\begin{align*}
	(\ddt^2 \ulod{n}, z_\f) + a(\ddt \ulod{n}, z_\f) + b(\ulod{n}, z_\f) &= (\ddt^2\ulod{n},z_\f).
	\end{align*}
	For $X^n$ we use \eqref{eq:Xproj1}-\eqref{eq:Xproj2} and the orthogonality to deduce
	\begin{align*}
	(\ddt^2 X^n, z) + a(\ddt X^n, z) + b(X^n, z) &= (\ddt^2 X^n, z) + (f^n-\ddt^2u^n_h,z_\ms), \quad z \in V_h,
	\end{align*}
	Hence, $\theta^n$ satisfies
	\begin{align*}
	(\ddt^2 \theta^n, z) + a(\ddt \theta^n, z) + b(\theta^n, z)  &= (-\ddt^2 \rho^n,z) - (\ddt^2 u^n_h, z_\f) \\ &\qquad+ (\ddt^2\ulod{n}, z_\f) + (\ddt^2 w^n, z_\ms), \quad z\in V_h,
	\end{align*}
	with $\theta^0=\theta^1=0$, since $\ulod{0} = u^0_h = X^0$ and $\ulod{1} = u^1_h = X^1$. 
	Let $\tilde \theta^n = \sum_{j=2}^n \tau \theta^j$. Multiplying by $\tau$ and summing over $n$ gives
	\begin{align}\label{eq_integrated}
	(\ddt \theta^n, z) + a(\theta^n, z) + b(\tilde \theta^n, z) & \leq (-\ddt \rho^n,z) - (\ddt u^n_h - \ddt u^1_h, z_\f) \\&\quad+ (\ddt \ulod{n} - \ddt \ulod{1}, z_\f)
	+ (\ddt w^n- \ddt w^1, z_\ms),\notag
	\end{align}
	where we have used that $\theta^1=\theta^0=\rho^1=\rho^0=0$.
	Using the interpolant $I_H$ we deduce
	\begin{align*}
	(\ddt u^n_h, z_\f) + (\ddt \ulod{n}, z_\f) + (\ddt w^n, z_\ms) &\leq CH(\|\ddt u^n_h\|_{L_2} + \|\ddt \ulod{n}\|_{L_2})\|z\|_{H^1} \\ &\qquad + CH\|\ddt \ulod{n}\|_{H^1}\|z_\ms\|_{L_2},
	\end{align*}
	for $1\leq n \leq N$. Let $\alpha(n) = \|\ddt u^n_h\|_{L_2} + \|\ddt \ulod{n}\|_{H^1}$. Since $\|z_\ms\|_{L_2}\leq C\|z\|_{H^1}$ and $\alpha(1)=0$ due to the vanishing initial data, we get
	\begin{align*}
	(\ddt \theta^n, z) + a(\theta^n, z) + b(\tilde \theta^n, z) & \leq (-\ddt \rho^n,z) + CH\alpha(n)\|z\|_{H^1}, \ z\in V_h.
	\end{align*}
	Now, choose $z = \theta^n = \ddt \tilde \theta^n$ in \eqref{eq_integrated}. We get
	\begin{align*}
	\frac{1}{2}\|\theta^n\|^2_{L_2} - \frac{1}{2}\|\theta^{n-1}\|^2_{L_2} + &\tau\|\theta^n\|^2_{a} + \frac{1}{2}\|\tilde \theta^n\|^2_b - \frac{1}{2}\|\tilde \theta^{n-1}\|^2_b \\&\leq \tau \|\ddt \rho^n\|_{L_2}\|\theta^n\|_{L_2} + CH\tau\alpha(n)\|\theta^n\|_{H^1}.
	\end{align*}
	Summing over $n$ gives
	\begin{align*}
	\|\theta^n\|^2_{L_2} + \sum_{j=2}^n \tau\|\theta^j\|^2_{H^1} + \|\tilde\theta^n\|^2_{H^1} \leq \sum_{j=2}^n \tau \|\ddt \rho^j\|_{L_2}\|\theta^j\|_{L_2} + CH\sum_{j=2}^n\tau\alpha(j)\|\theta^j\|_{H^1}.
	\end{align*}
	Now using that $\|\theta^n\|_{L_2}\leq \|\theta^n\|_{H^1}$ and Young's weighted inequality, $\theta^j$ can be kicked back to the left hand side. We deduce
	\begin{align*}
	\sum_{j=2}^n \tau\|\theta^j\|^2_{H^1} &\leq C\sum_{j=2}^n \tau \|\ddt \rho^j\|^2_{L_2} + CH^2\sum_{j=2}^n\tau\alpha(j)^2.
	\end{align*}
	Using Lemma~\ref{RitzVolterra_err_timederivative} we have
	\begin{align*}
	\sum_{j=2}^n \tau\|\theta^j\|^2_{H^1} &\leq CH^2\Bigg(\sum_{j=2}^n \tau(\|\ddt f^j\|^2_{L_2} + \|\ddt^2 u^j_h\|^2_{L_2}) + \|\ddt^2 u^2_h\|^2_{L_2}\Bigg)  + CH^2\sum_{j=2}^n\tau\alpha(j)^2.
	\end{align*}
	To bound $\|\ddt^2 u^{2}_h\|^2_{L_2}$, we consider \eqref{eq:fullydiscrete} for $n=2$ and choose $v=\ddt^2u^2_h$, which gives
	\begin{align*}
	(\ddt^2 u^2_h, \ddt^2 u^2_h) + a(\ddt u^2_h, \ddt^2 u^2_h) + b(u^2_h, \ddt^2 u^2_h) = (\ddt f^2,\ddt^2 u^2_h).
	\end{align*}
	Due to the vanishing initial data $\ddt u^2_h = \tau^{-1}u^2_h$ and $\ddt^2u^2_h = \tau^{-2}u^2_h$. We get
	\begin{align}\label{n2_bound}
	\|\ddt^2 u^2_h\|^2_{L_2} + \frac{1}{\tau^3}\|u^2_h\|^2_a + \frac{1}{\tau^2}\|u^2_h\|^2_b = (f^2, \ddt^2 u^2_h),
	\end{align} 
	and we deduce
	\begin{align*}
	\|\ddt^2 u^2_h\|^2_{L_2} \leq  C\|f^2\|^2_{L_2}.
	\end{align*}
	
	All terms, except $\sum_{j=2}^n\tau\|\ddt \ulod{j}\|^2_{H^1}$ that appears in $\sum_{j=2}^n\alpha^2(j)$, can now be bounded by using the regularity in Theorem~\ref{thm:energy_bounds}. To bound $\sum_{j=2}^n\tau\|\ddt \ulod{j}\|^2_{H^1}$ we choose $z=\ddt v^n$ and $z=\ddt w^n$ in \eqref{eq_GFEM_full_1} and \eqref{eq:GFEM_full_2} respectively. Adding the two equations and using the orthogonality between $V_\ms$ and $V_\f$ we achieve
	\begin{align*}
	(\ddt^2v^n,\ddt v^n) + a(\ddt \ulod{n}, \ddt \ulod{n}) + b(\ulod{n},\ddt \ulod{n}) &= (f^n, \ddt v^n)\\& \leq C_\epsilon \|f^n\|^2_{L_2} + \epsilon\|\ddt v^n\|^2_{L_2}.
	\end{align*}
	Note that $\|\ddt v^n\|_{L_2} \leq C\|\nabla \ddt v^n\|_{L_2} \leq C\vertiii{\ddt v^n} = C \vertiii{\ddt \ulod{n}} \leq C \|\ddt \ulod{n}\|_a$ , so we may choose $\epsilon$ small enough such that $\vertiii{\ddt \ulod{n}}$ can be kicked to the left hand side. As in the proof of Theorem~\ref{thm:energy_bounds} we may now deduce
	\begin{align}\label{ulod_energy_bound}
	\|\ddt v^n\|^2_{L_2} + \sum_{j=2}^n \tau \|\ddt \ulod{j}\|^2_{H^1} + \|\ulod{n}\|^2_{H^1} \leq C\Bigg(\sum_{j=2}^n \|f^j\|^2_{L_2} + \|\ddt u^1_h\|^2_{L_2} + \|u^1_h\|^2_{H^1}\Bigg),
	\end{align}
	where we have used that $v^1 = \ulod{1} = u^1_h \in V_\ms$. However, we have assumed vanishing initial data so these terms disappear. The lemma follows.
\end{proof}

\begin{lemma}\label{lemma_error_zero_rhs}
	Let $u^n_h$ and $\ulod{n}$ be the solutions to \eqref{eq:fullydiscrete} and \eqref{eq_GFEM_full_1}-\eqref{eq:GFEM_full_2}, respectively. Assume that $f=0$. The error is bounded by
	\begin{align*}
	\sum_{j=2}^n\tau t_j^2\|\ulod{j} - u^j_h\|^2_{H^1} \leq C H^2(\|\ddt u^1_h\|^2_{H^1} + \|u^1_h\|^2_{H^1} + \|u^0_h\|^2_{H^1}), \quad n \geq 2,
	\end{align*}
	where C does not depend on the variations in $A$ or $B$.
\end{lemma}

\begin{proof}
	We follow the steps in the proof of Lemma~\ref{lemma_error_zero_initial_data} to equation \eqref{eq_integrated}. Note that $\|\rho^n\|_{H^1}$ can be bounded by Lemma~\ref{RitzVolterra_err} and the energy bound in \eqref{energy_bound2} with $f=0$. 
	
	Now, let $\tilde \theta^n = \sum_{j=2}^n \tau \theta^j$. Choose $z = \theta^n = \ddt \tilde \theta^n$ in \eqref{eq_integrated} and multiply by $\tau t_n^2$. We get 
	\begin{align*}
	\frac{t_n^2}{2}\|\theta^n\|^2_{L_2} &- \frac{t_{n-1}^2}{2}\|\theta^{n-1}\|^2_{L_2} + \tau t_n^2\|\theta^n\|^2_{a} + \frac{t_n^2}{2}\|\tilde \theta^n\|^2_b - \frac{t_{n-1}^2}{2}\|\tilde \theta^{n-1}\|^2_b \\&\leq \tau t_n^2\|\ddt \rho^n\|_{L_2}\|\theta^n\|_{L_2} + CHt_n^2\tau(\alpha(n)+\alpha(1))\|\theta^n\|_{H^1}
	\\&\qquad \qquad+ \frac{(t_n^2-t_{n-1}^2)}{2}\|\theta^{n-1}\|^2_{L_2} + \frac{(t_n^2-t_{n-1}^2)}{2}\|\tilde \theta^{n-1}\|^2_b.
	\end{align*}
	Summing over $n$ and using $t_n^2-t_{n-1}^2 \leq 2\tau t_n$ gives
	\begin{align}\label{main_eq}
	t_n^2\|\theta^n\|^2_{L_2} + &\sum_{j=2}^n \tau t_j^2\|\theta^j\|^2_{H^1} + t_n^2\|\tilde \theta^n\|^2_{H^1} \leq C\sum_{j=2}^n \tau t_j^2\|\ddt \rho^j\|_{L_2}\|\theta^j\|_{L_2} \\&\quad+ CH\sum_{j=2}^n\tau t_j^2(\alpha(j)+\alpha(1))\|\theta^j\|_{H^1} 
	+ C\sum_{j=2}^n \tau t_j\|\theta^j\|^2_{L_2} + C\sum_{j=2}^n \tau t_j\|\tilde \theta^j\|^2_b. \notag
	\end{align}
	From the first two sums on the right hand side we can kick $t_j\|\theta^j\|_{L_2}\leq t_j\|\theta^j\|_{H^1}$ and $t_j\|\theta^j\|_{H^1}$ to the left hand side. The remaining two sums needs to be bounded by other energy estimates.
	
	Multiply \eqref{eq_integrated} by $\tau$ and sum over $n$ to get
	\begin{align}\label{eq_integrated_twice}
	(\theta^n, z) + &a(\tilde \theta^n, z) + b\Bigg(\sum_{j=2}^n \tau \tilde \theta^j, z\Bigg) \leq (\rho^n,z) - (u^n_h - u^1_h, z_\f) + (\ulod{n} - \ulod{1}, z_\f) \\
	&\quad+ ( w^n- w^1, z_\ms)  + t_n((\ddt u^1_h,z_\f) - (\ddt \ulod{1},z_\f) - (\ddt w^1,z_\ms)).\notag
	\end{align}
	where we have used $\theta^1 = \rho^1 = 0$. As in the proof of Lemma~\ref{lemma_error_zero_initial_data} we get
	\begin{align*} 
	(u^n_h, z_\f) + (\ulod{n}, z_\f) + (w^n, z_\ms) &\leq CH(\|u^n_h\|_{L_2} + \|\ulod{n}\|_{L_2})\|z\|_{H^1} \\&\qquad+ CH\| \ulod{n}\|_{H^1}\|z_\ms\|_{L_2},
	\end{align*}
	for $1 \leq n \leq N$. Let $\beta(n) = \|u^n_h\|_{L_2} + \|\ulod{n}\|_{H^1}$. Choose $z = \tilde \theta^n = \ddt \sum_{j=1}^n\tau\tilde \theta^j$. Similar to above energy estimates, we get
	\begin{align}\label{eq_integrated_twice_1}
	\|\tilde \theta^n\|^2_{L_2} + \sum_{j=2}^n\tau\|\tilde \theta^j\|^2_a + \|\sum_{j=2}^n \tau \tilde \theta^j\|^2_b  \leq \sum_{j=2}^n \tau \|\rho^j\|^2_{L_2} + CH^2\sum_{j=2}^n\tau (\beta(j) + \beta(1) + \alpha(1))^2.
	\end{align}
	Since
	$
	\sum_{j=2}^n \tau t_j\|\tilde \theta^j|^2_b \leq C(t_n)\sum_{j=2}^n \tau \|\tilde \theta^j\|^2_a
	$ 
	we may use \eqref{eq_integrated_twice_1} in \eqref{main_eq}. This gives
	\begin{align}\label{main_eq_2}
	t_n^2&\|\theta^n\|^2_{L_2} + \sum_{j=2}^n \tau t_j^2\|\theta^j\|^2_{H^1} + t_n^2\|\tilde \theta^n\|^2_{H^1} \leq C\sum_{j=2}^n \tau (t_j^2\|\ddt \rho^j\|^2_{L_2}+\|\rho^j\|^2_{L_2}) \\&\quad+ C\sum_{j=2}^n \tau t_j\|\theta^j\|^2_{L_2} 
	+CH^2\sum_{j=2}^n\tau (t_j^2(\alpha(j)+\alpha(1))^2 + (\beta(j) + \beta(1) + \alpha(1))^2). \notag
	\end{align}
	It remains to bound $C\sum_{j=2}^n \tau t_j\|\theta^j\|^2_{L_2}$. For this purpose, choose $z = \theta^n = \ddt \tilde \theta^n$ in \eqref{eq_integrated_twice}. Multiply by $t_n\tau$ and sum over $n$ to achieve
	\begin{align}\label{eq_integrated_twice_2}
	\sum_{j=2}^n \tau t_j\|\theta^j\|^2_{L_2} + t_n\|\tilde \theta^n\|^2_a + &\sum_{j=2}^n t_j\tau b\Bigg(\sum_{k=2}^j \tau \tilde \theta^k, \ddt \tilde \theta^j\Bigg)  \leq \sum_{j=1}^n \tau t_j\|\rho^j\|^2_{L_2} + \sum_{j=2}^n \tau \|\tilde \theta^j\|^2_a \\&\quad + CH\sum_{j=2}^n\tau t_j(\beta(j) + \beta(1) + \alpha(1))\|\theta^j\|_{H^1}.\notag
	\end{align}
	Note that $\|\theta^j\|_{H^1}$ in the last sum in only present in the right hand side. The second term on the right hand side is bounded by \eqref{eq_integrated_twice_1}. For the term involving the bilinear form $b(\cdot,\cdot)$ we use summation by parts to get
	\begin{align*}
	-\sum_{j=2}^n \tau b\Bigg(t_j\sum_{k=2}^j \tau \tilde \theta^k, \ddt \sum_{k=2}^j \tau \theta^k\Bigg) &\leq \sum_{j=2}^n \tau b\Bigg(t_j\tilde \theta^j + \sum_{k=2}^j \tau \tilde \theta^k, \tilde \theta^{j-1}\Bigg)  - b\Bigg(t_n\sum_{j=2}^n \tau \tilde \theta^j, \tilde \theta^n\Bigg) \\&\leq C\sum_{j=2}^n \tau t_j\|\tilde \theta^j\|^2_b 
	+ C\|\sum_{j=2}^n \tau \tilde \theta^j\|^2_b + C_\epsilon t_n^2\|\tilde \theta^n\|^2_{H^1}.
	\end{align*}
	Here the constant $C_\epsilon$ can be made arbitrarily small due to Young's weighted inequality. The first two terms can be bounded by \eqref{eq_integrated_twice_1}. Thus, \eqref{eq_integrated_twice_2} becomes
	\begin{align*}
	\sum_{j=2}^n \tau t_j\|\theta^j\|^2_{L_2} + t_n\|\tilde \theta^n\|^2_a &\leq \sum_{j=1}^n \tau( t_j\|\rho^j\|^2_{L_2} + \|\rho^j\|^2_{L_2}) + CH\sum_{j=2}^n\tau t_j(\beta(j) + \beta(1)\\& + \alpha(1))\|\theta^j\|_{H^1} + CH^2\sum_{j=2}^n\tau (\beta(j) + \beta(1) + \alpha(1))^2.
	\end{align*}
	Using this in \eqref{main_eq_2} we arrive at
	\begin{align}\label{final_bound}
	t_n^2\|\theta^n\|^2_{L_2} + \sum_{j=2}^n \tau t_j^2&\|\theta^j\|^2_{H^1} + t_n^2\|\tilde \theta^n\|^2_{H^1} \leq C\sum_{j=2}^n \tau (t_j^2\|\ddt \rho^j\|^2_{L_2} + t_j\|\rho^j\|^2_{L_2}) \notag\\&\quad +
	CH^2\sum_{j=2}^n\tau (t_j^2(\alpha(j)+\alpha(1))^2 + (\beta(j) + \beta(1) + \alpha(1))^2)
	\\&\quad +  CH\sum_{j=2}^n\tau t_j(\beta(j) + \beta(1) + \alpha(1))\|\theta^j\|_{H^1}
	+ C_\epsilon t_n^2\|\tilde \theta^n\|^2_{H^1}\notag.
	\end{align}
	Using Lemma~\ref{RitzVolterra_err_timederivative} and Lemma~\ref{RitzVolterra_err} with $f=0$ we deduce for the first two terms in \eqref{final_bound}
	\begin{align*}
	C\sum_{j=2}^n \tau (t_j^2\|\ddt \rho^j\|^2_{L_2} + t_j\|\rho^j\|^2_{L_2}) &\leq CH^2 \Bigg(\sum_{j=2}^n \tau  \|\ddt^2u^j_h\|^2_{L_2} + t^2_2\|\ddt^2u^2_h\|^2_{L_2}\Bigg),
	\end{align*}
	where we can use \eqref{energy_bound2} for $n=2$ and $f=0$ to bound $\ddt^2u^2_h$. We get
	\begin{align*}
	t_2^2\|\ddt^2u^2_h\|^2_{L_2} \leq C\tau(\|\ddt u^1_h\|^2_{H^1} + \|u^1_h\|^2_{H^1}).
	\end{align*}
	For the remaining terms in \eqref{final_bound} we note that $t_j\|\theta^j\|_{H^1}$ now may be kicked to left hand side using Cauchy--Schwarz and Young's weighted inequality. The term involving $C_\epsilon$ can also be moved to the left hand side. All terms involving $\alpha(j)$ and $\beta(j)$ can be bounded by \eqref{energy_bound1} and \eqref{ulod_energy_bound}. This finishes the proof after using the regularity in Theorem~\ref{thm:energy_bounds} with $f=0$.
\end{proof}

\subsection{Error bound for the ideal method}

We get the final result by combining the two previous lemmas.

\begin{corollary}
	\label{cor:final}
	Let $u^n_h$ and $\ulod{n}$ be the solutions to \eqref{eq:fullydiscrete} and \eqref{eq_GFEM_full_1}-\eqref{eq:GFEM_full_2}, respectively. The solutions can be split into $u^n_h = u^n_{h,1} + u^n_{h,2}$ and $\ulod{n} = u^n_{\mathrm{lod},1} + u^n_{\mathrm{lod},2}$, where the first part has vanishing initial data, and the second part a vanishing right hand side. The error is bounded by
	\begin{align*}
	\sum_{j=2}^n\tau\|u^j_{h,1} - u^j_{\mathrm{lod},1}\|^2_{H^1} \leq C H^2 \Bigg( \sum_{j=1}^n \tau (\|f^j\|^2_{L_2} + \|\ddt f^j\|^2_{L_2}) + \max_{j=1,...,n} \|f^j\|^2_{L_2}\Bigg), 
	\end{align*}
	and
	\begin{align*}
	\sum_{j=2}^n\tau t_j^2\|u^j_{h,2} - u^j_{\mathrm{lod},2}\|^2_{H^1} \leq C H^2(\|\ddt u^1_h\|^2_{H^1} + \|u^1_h\|^2_{H^1} + \|u^0_h\|^2_{H^1}), 
	\end{align*}
	for $n\geq 2$.
\end{corollary}
\begin{proof}
	This is a direct consequence of Lemma~\ref{lemma_error_zero_initial_data} and Lemma~\ref{lemma_error_zero_rhs} together with the fact that the problem is linear so the error can be split into two contributions satisfying the conditions of each lemma.
\end{proof}

\begin{remark}
	\label{rmk:loc}
	The result from Corollary \ref{cor:final} is derived for the ideal method presented in \eqref{eq_GFEM_full_1}-\eqref{eq:GFEM_full_2}. The GFEM in \eqref{eq:xink}-\eqref{eq_GFEM_full_loc_1} will yield yet another error from the localization procedure. However, due to the exponential decay in Theorem \ref{thm:expdecay} and Theorem \ref{thm:expdecay2}, it holds for the choice $k\approx |\log(H)|$ that the perturbation from the ideal method is of higher order and the derived result in Corollary \ref{cor:final} is still valid. For the details regarding the error from the localization procedure, we refer to \cite{MalqvistPeterseim}.
\end{remark}

\subsection{Initial data}
\label{sec:initialdata}

For general initial data $u^0_h, u^1_h \in V_h$ we consider the projections $R_\ms u^0_h$ and $R_\ms u^1_h$, where $R_\ms = I - R_\f$ is the Ritz-projection onto $V_\ms$. Let $v$ be the difference between two solutions to the damped wave equation with the different initial data. From \eqref{energy_bound1} it follows that 
\begin{align*}
\|v\|^2_{H^1} \leq C(\|\ddt (u^1_h - R_\ms u^1_h)\|^2_{L_2} + \|u^1_h - R_\ms u^1_h\|^2_b), 
\end{align*}
where we have chosen to keep the $b$-norm. For the first term we may use the interpolant $I_H$ to achieve $H$. For the second term use
\begin{align}\label{data_constant}
\|u^1_h - R_\ms u^1_h\|^2_b \leq \frac{\beta_+}{\alpha_- + \tau \beta_-}(\|u^1_h - R_\ms u^1_h\|^2_a + \tau \|u^1_h - R_\ms u^1_h\|^2_b),
\end{align}
If the initial data fulfills the following condition for some $g\in L_2(\Omega)$
\begin{align}\label{well_prepared_data}
a(u^1_h,v) + \tau b(u^1_h,v) = (g,v), \quad \forall v \in V_h,
\end{align}
then we may deduce
\begin{align*}
\|u^1_h - R_\ms u^1_h\|^2_a + \tau \|u^1_h - R_\ms u^1_h\|^2_b = (g,u^1_h-R_\ms u^1_h) \leq CH\|g\|_{L_2}\|u^1_h-R_\ms u^1_h\|_{H^1}.
\end{align*}
Hence, the error introduced by the projection of the initial data is of order $H$. The condition \eqref{well_prepared_data} appears when applying the LOD method to classical wave equations, see \cite{Wave}, where it is referred to as ``well prepared data''. We note in our case that if $B$ is small compared to $A$, that is if the damping is strong, then the constant in \eqref{data_constant} is small. In some sense, this means that the condition in \eqref{well_prepared_data} is of ``less importance'', which is consistent with the fact that strong damping reduces the impact of the initial data over time.

\section{Reduced basis method}
\label{sec:rb}

The GFEM as it is currently stated requires us to solve the system in \eqref{eq:xink} for each coarse node in each time step, i.e.\ $N$ number of times. We will alter the method by applying a reduced basis method, such that it will suffice to find the solutions for $M < N$ time steps, and compute the remaining in a significantly cheaper and efficient way.

First of all, we note how the system \eqref{eq:xink} that $\xi^n_{k,x}$ solves resembles a parabolic type equation with no source term. That is, the solution will decay exponentially until it is completely vanished. An example of how $\xi^n_{k,x}$ vanish with increasing $n$ can be seen in Figure \ref{fig:DecayingCorrection}, where the coefficients are given as 
\begin{equation*}
A(x) = \Big(2-\sin\big(\tfrac{2\pi x}{\varepsilon_A}\big)\Big)^{-1} \ \text{and}  \quad B(x) = \Big(2-\cos\big(\tfrac{2\pi x}{\varepsilon_B}\big)\Big)^{-1},
\end{equation*}
with $\varepsilon_A = 2^{-4}$ and $\varepsilon_B = 2^{-6}$.

\begin{figure}[b!]
	\centering
	\includegraphics[width=10cm]{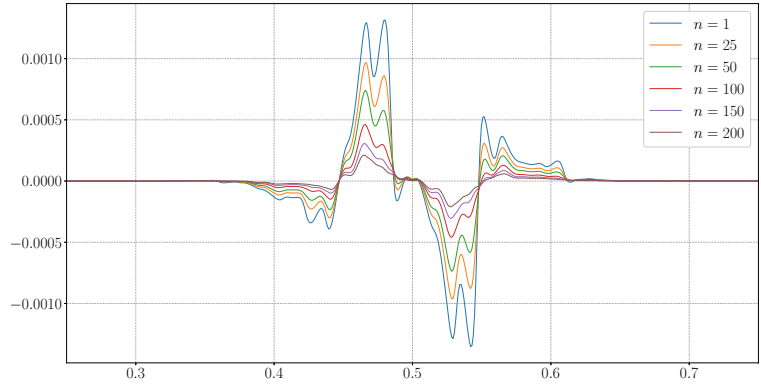}
	\caption{The behavior of the correction functions $\xi^n_{k,x}$ with increasing $n$. The time step is $\tau= 0.01$ and $k$ is here chosen so that the support covers the entire interval.}
	\label{fig:DecayingCorrection}
\end{figure}

In Figure \ref{fig:DecayingCorrection} it is also seen how the solutions decay with a similar shape through all time steps. This gives the idea that it is possible to only evaluate the solutions for a few time steps, and utilize these solutions to find the remaining ones. This idea can be further investigated by storing the solutions $\{\xi^n_{k,x}\}_{n=1}^N$ and analyzing the corresponding singular values. The singular values are plotted and seen in Figure \ref{fig:SingularValues}. It is seen how the values decrease rapidly, and that most of the values lie on machine precision level. In practice, this means that the information in $\{\xi^n_{k,x}\}_{n=1}^N$ can be extracted from only a few $\xi^n_{k,x}$'s. We use this property to decrease the computational complexity by means of a reduced basis method. We remark that singular value decomposition is not used for the method itself, but is merely used as a tool to analyze the possibility of applying reduced basis methods.

\begin{figure}
	\centering
	\includegraphics[width=10cm]{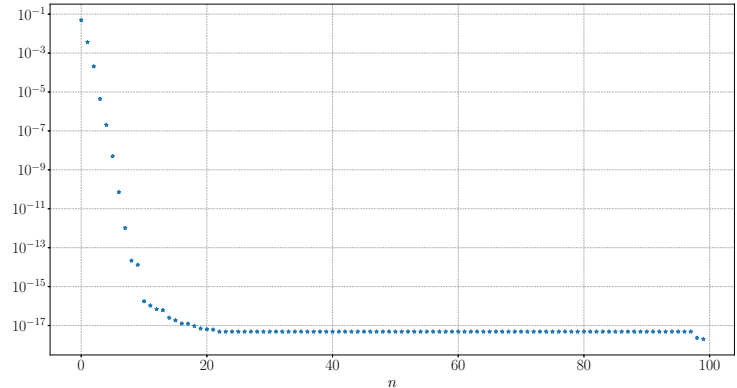}
	\caption{The singular values obtained when performing a singular value decomposition of the matrix created by storing the finescale corrections $\{\xi^n_{k,x}\}_{n=1}^N$ with $N=100$.}
	\label{fig:SingularValues}
\end{figure}

The main idea behind reduced basis methods is to find an approximate solution in a low-dimensional space $V^{\mathrm{RB}}_{M,k,x}$, which is created using a number of already computed solutions. More precisely, to construct a basis for this space, one first computes $M$ solutions $\{\xi^m_{k,x}\}_{m=1}^M$, where $M < N$. By orthonormalizing these solutions using e.g.\ Gram--Schmidt orthonormalization, we yield a set of vectors $\{\zeta^m_{k,x}\}_{m=1}^M$, called the reduced basis. Consequently, the reduced basis space becomes $V^{\mathrm{RB}}_{M,k,x} = \text{span}(\{\zeta^m_{k,x}\}_{m=1}^M)$ for each node $x\in \mathcal{N}$. With this space created, the procedure of finding $\{\xi^n_{k,x}\}_{n=1}^N$ is now reduced to finding $\{\xi^n_{k,x}\}_{n=1}^M$, and then approximate the remaining solutions by $\{\xi^{n,\mathrm{rb}}_{k,x}\}_{n=M+1}^N \subset V^{\mathrm{RB}}_{M,k,x}$. The matrix system to solve for a solution in $V^{\mathrm{RB}}_{M,k,x}$ is of dimension $M\times M$, so when $M$ is chosen small, the last $N-M$ solutions are significantly cheaper to compute, which solves the issue of computing $N$ problems on the finescale space.

When constructing the reduced basis $\{\zeta^m_{k,x}\}_{m=1}^M$, it is important to be aware of the fact that the solution corrections $\{\xi^n_{k,x}\}_{n=1}^N$ all show very similar behavior. In practice, this implies that many of the $\xi^n_{k,x}$'s are linearly dependent, hence causing floating point errors to become of significant size in the RB-space $V^{\mathrm{RB}}_{M,k,x}$. To work around this issue, one may include a relative tolerance level that removes a vector from the basis if it is too close to being linearly dependent to one of the previously orthonormalized vectors. One may moreover use this tolerance level as a criterion for the amount of solutions, $M$, to pre-compute. That is, once the first vector is removed from the orthonormalization process, then the RB-space contains sufficient information and no more solutions need to be added.

In total, the novel method first requires that we solve $N_H$ number of systems on the localized fine scale in order to construct the multiscale space $V_{\ms,k}$. Moreover, we require to solve a localized fine system $N_H$ times for $M$ time steps to create the RB-space $V^{\mathrm{RB}}_{M,k,x}$ for each coarse node $x\in \mathcal{N}$. By utilizing the RB-space, the remaining $N-M$ finescale corrections are then solved for in an $M\times M$ matrix system, and we yield the sought solution $u^{N, \mathrm{rb}}_{\mathrm{lod}, k}$ by computing a matrix system on the coarse grid with the multiscale space $V_{\ms,k}$.

\section{Numerical examples}
\label{sec:numexamples}

In this section we present numerical examples that illustrate the performance of the established theory. For all examples, we consider the domain to be the unit square $\Omega = [0,1] \times [0,1]$. The coefficients $A(x,y)$ and $B(x,y)$ used in all examples are generated randomly with values in the interval $[10^{-1},10^{3}]$, and examples of such are seen in Figure \ref{fig:coefficients}. Moreover, as initial value for each example we set $u_0 = u_1 = 0$, and the source function is given by $f = 1$.

\begin{figure}
    \centering
    \begin{subfigure}[b]{0.4\textwidth}
        \includegraphics[width=\textwidth]{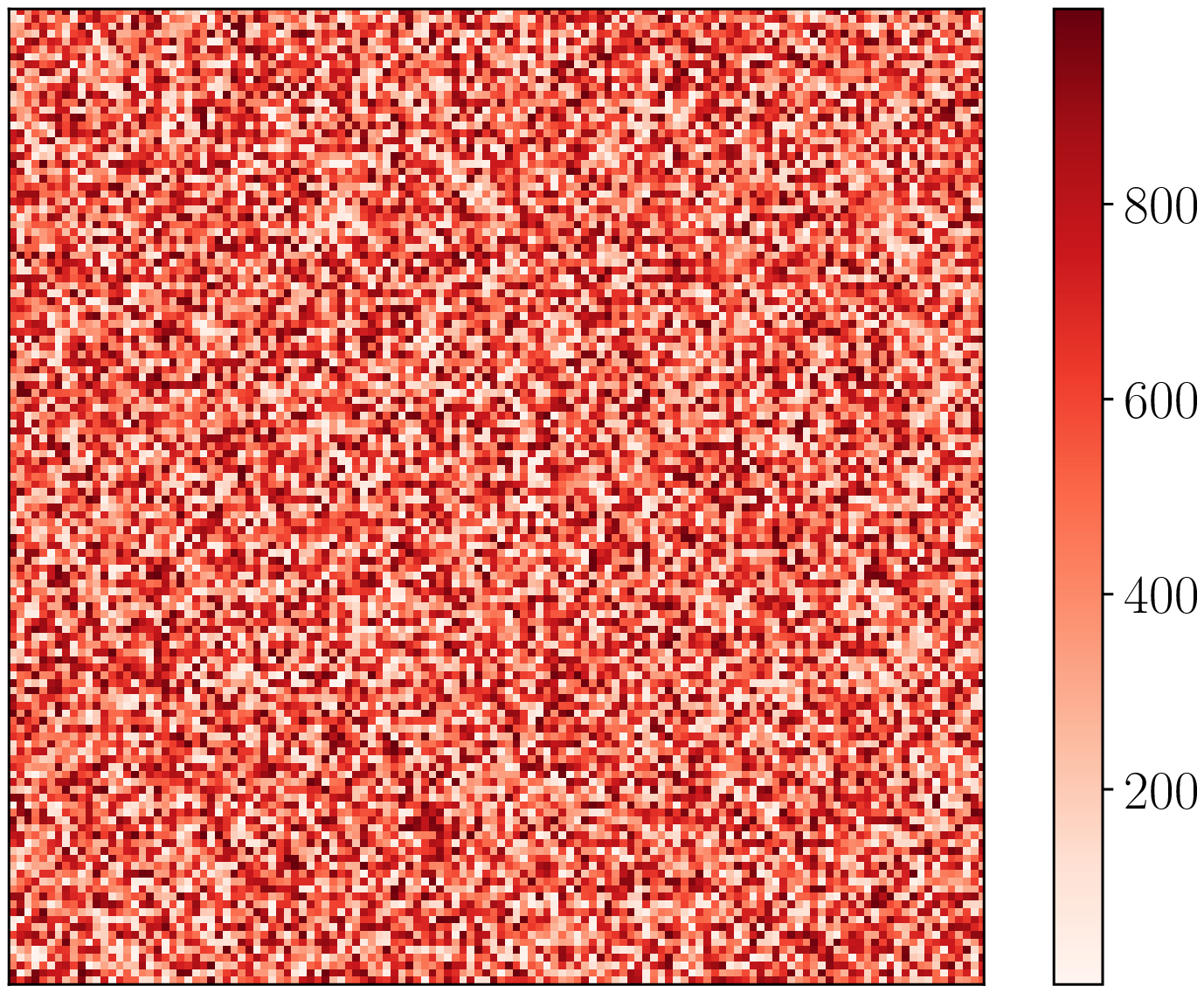}
        \caption{$A(x,y)$.}
        \label{fig:gull}
    \end{subfigure}
    ~ 
    \begin{subfigure}[b]{0.4\textwidth}
        \includegraphics[width=\textwidth]{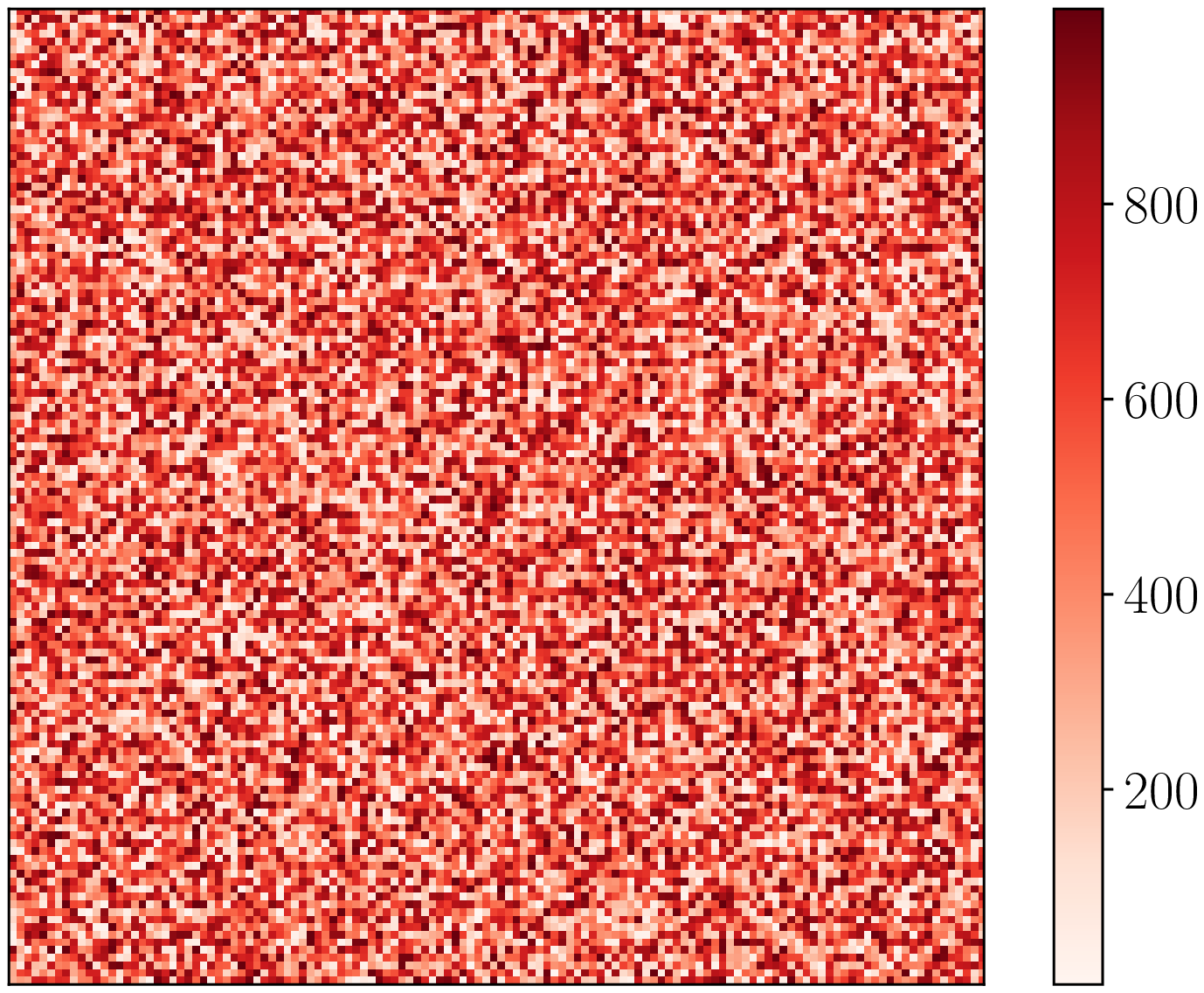}
        \caption{$B(x,y)$.}
        \label{fig:tiger}
    \end{subfigure}
    \caption{The two different coefficients used for the numerical examples. The contrast is $\alpha_+/\alpha_-= \beta_+/\beta_- = 10^4$.}\label{fig:coefficients}
\end{figure}

The first example is used to show how the performance is effected by the localization parameter $k$. Here, we evaluate the solution on the full grid, $\ulod{n}$, and compare it with the localized solution, $\ulodk{n}$, as $k$ varies. For the example the time step $\tau = 0.02$ was used and final time was set to $T = 1$. The fine and coarse meshes were set to $h = 2^{-7}$ and $H = 2^{-4}$ respectively, and we let $k=2,3,...,7$. The relative error between the functions can be seen in Figure \ref{fig:errork}. Here we can see how the error decays exponentially as $k$ increases, verifying the theoretical findings regarding the localization procedure.

For the second example, the performance of the GFEM in \eqref{eq:xink}-\eqref{eq_GFEM_full_loc_1} depending on the coarse mesh width $H$ is shown. For this example, the fine mesh width is set to $h=2^{-8}$, and for each coarse mesh width the localization parameter is set to $k=\log_2(1/H)$. Moreover, the time step is set to $\tau = 0.02$ (for the GFEM as well as the reference solution) and the solution is evaluated at $T=1$. To compute the error, we use a FEM solution on the fine mesh as a reference solution. The error as a function of $1/H$ can be seen in Figure \ref{fig:errorconvergence}. Here it is seen how the error for the novel method decays faster than linearly, confirming the error estimates derived in Section \ref{sec:errorestimates}. For comparison, Figure \ref{fig:errorconvergence} also shows the error of the standard FEM solution, as well as the solution using the standard LOD method with correction solely on $A$ and $B$ respectively, i.e.\ corrections based on the bilinear forms $a(\cdot,\cdot)$ and $b(\cdot,\cdot)$ respectively and without finescale correctors. As expected, the error of these methods stay at a constant level through all coarse grid sizes.

\begin{figure}
	\centering
	\includegraphics[width=11cm]{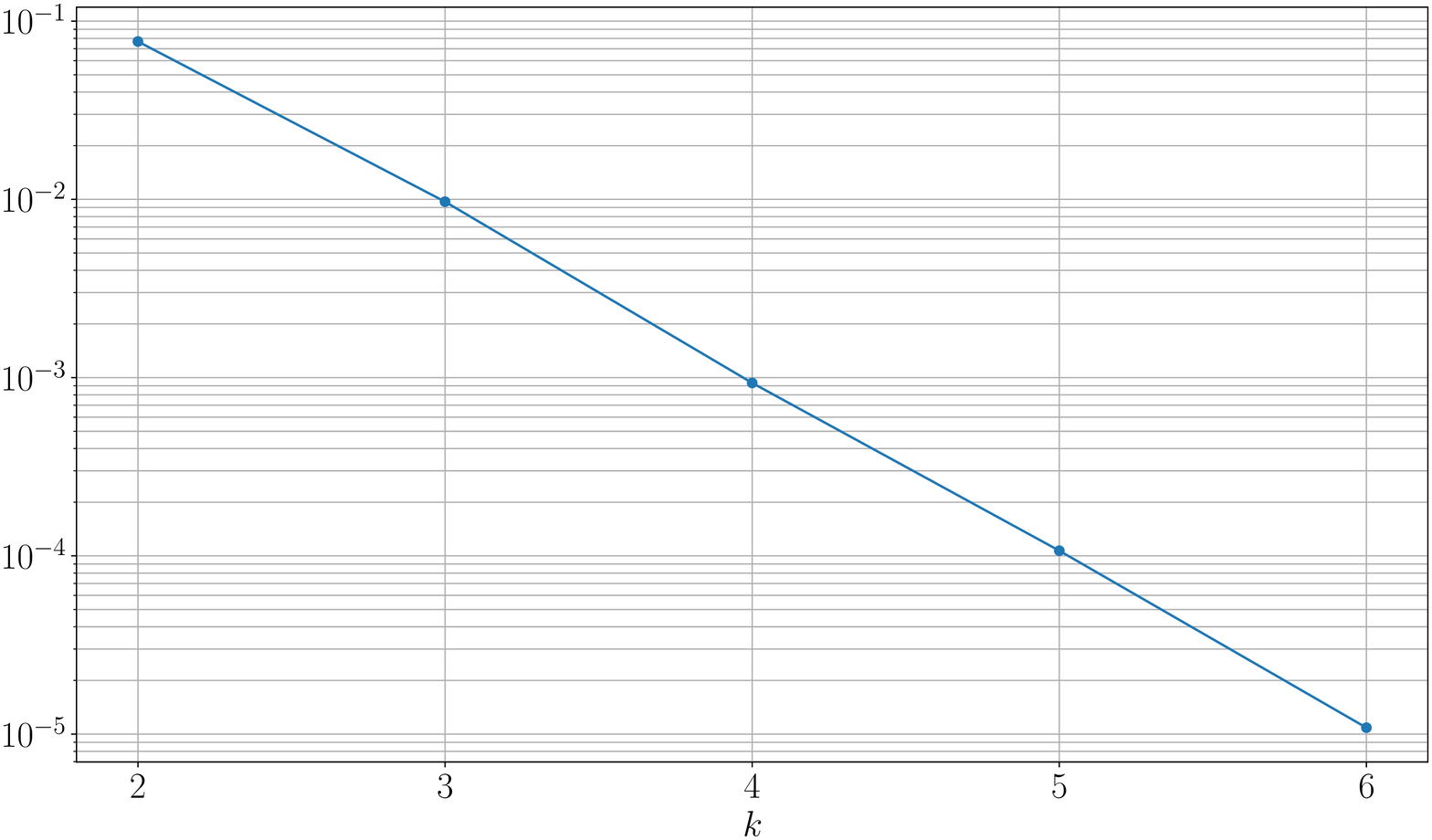}
	\caption{Relative $H^1$-error $\|\ulod{n} - \ulodk{n}\|_{H^1}/\|\ulod{n}\|_{H^1}$ between the non-localized and localized method, plotted against the layer number $k$.}
	\label{fig:errork}
\end{figure}

\begin{figure}
	\centering
	\includegraphics[width=11cm]{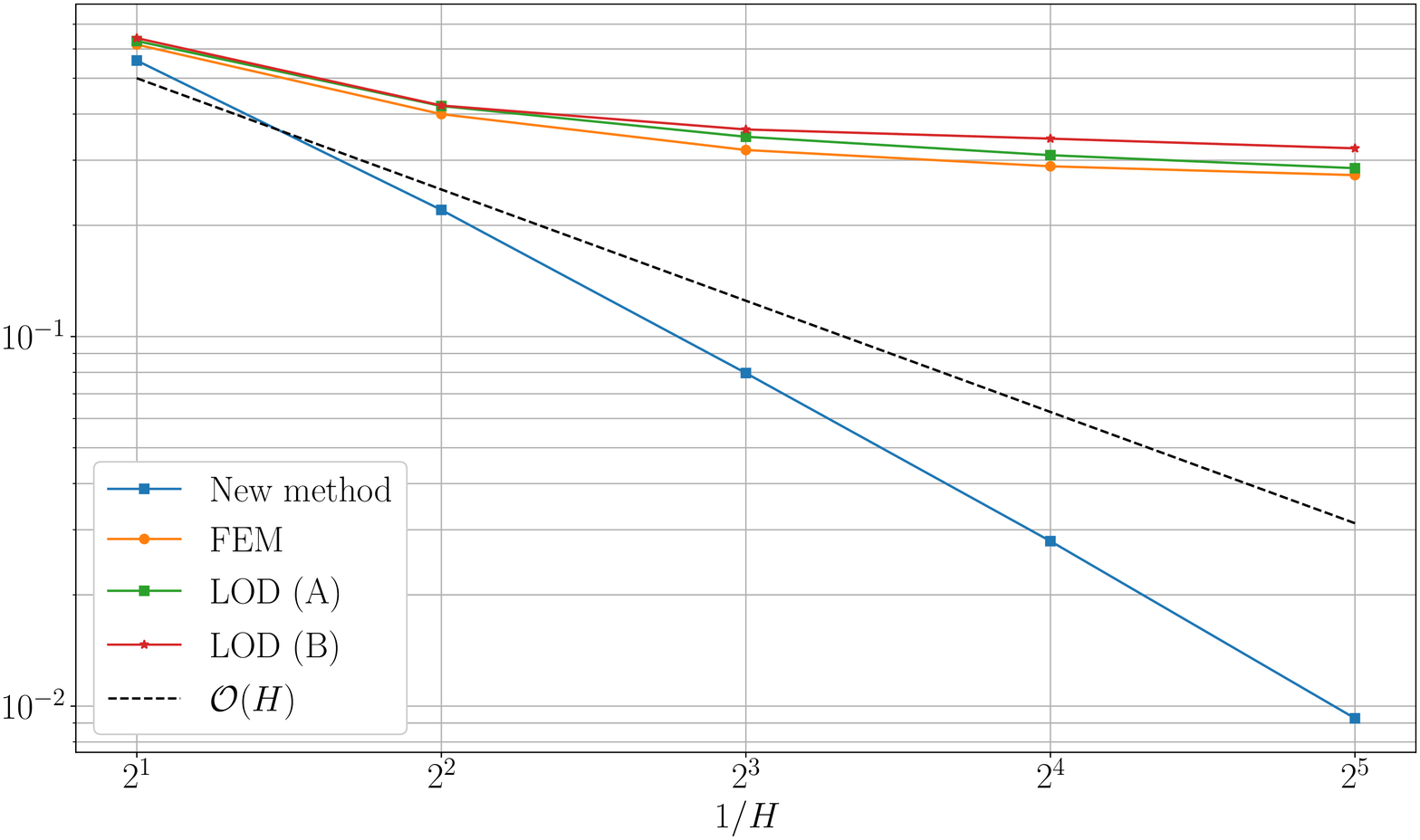}
	\caption{Relative $H^1$-error $\|u^{n}_{\mathrm{ref}} - \ulodk{n}\|_{H^1}/\|u^n_{\mathrm{ref}}\|_{H^1}$ between the reference solution and the approximate solution computed with the proposed method (without reduced basis computations).}
	\label{fig:errorconvergence}
\end{figure}

At last, we compute the solution where the system \eqref{eq:xink} is computed using the reduced basis approach. For this example, we let the number of pre-computed solutions $M$ vary, and see how the error between the solutions $\ulodk{n}$ and $u^{n,\mathrm{rb}}_{\mathrm{lod},k}$ behaves. In the example we have the fine mesh $h=2^{-8}$, the coarse mesh $H=2^{-5}$, the time step $\tau = 0.02$, and the final time $T=1$. The result can be seen in Figure \ref{fig:errorrb}. Here it is seen how the error decreases rapidly with the amount of pre-computed solutions. Note that it is sufficient to compute approximately 10 solutions to yield an error smaller than the discretization error for the main method in Figure \ref{fig:errorconvergence}. This for the case when the number of time steps are $N=50$. We emphasize that a large increment in time steps does not impact the number of pre-computed solutions $M$ significantly, making the RB-approach relatively more efficient the more time steps that are considered.

\begin{figure}
    \centering
    \includegraphics[width=11cm]{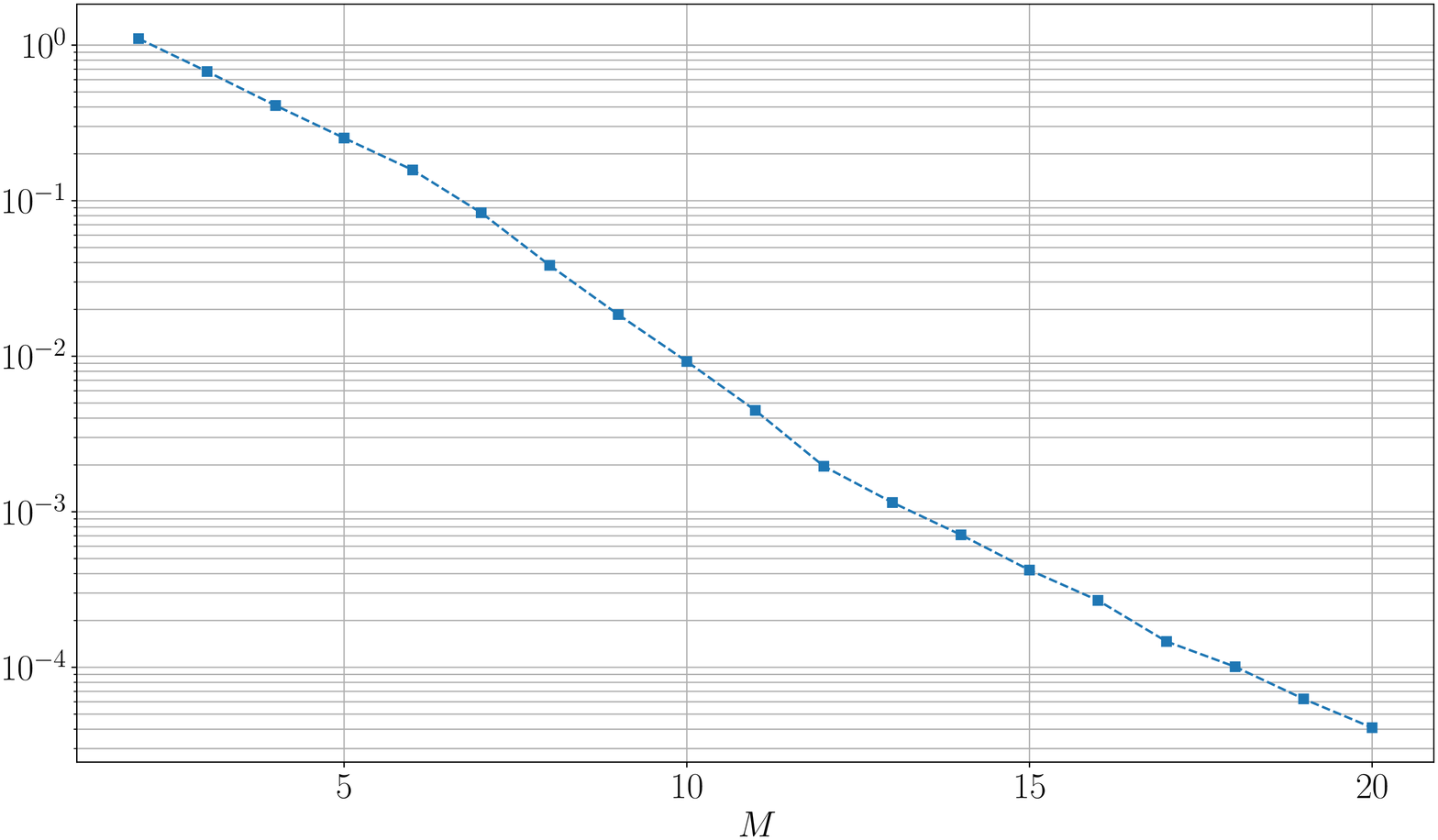}
    \caption{Relative $H^1$-error $\|\ulodk{n} - u^{n, \mathrm{rb}}_{\mathrm{lod}, k}\|_{H^1}/\|\ulodk{n}\|_{H^1}$ between the solution with and without the reduced basis approach, plotted against the number of pre-computed solutions.}
    \label{fig:errorrb}
\end{figure}

\bibliographystyle{abbrv}
\bibliography{References}

\end{document}